\documentclass[12pt]{article}

\usepackage[a4paper, total={6.2in, 8.2in}]{geometry}

\usepackage[numbers]{natbib}

\usepackage{mathtools}

\usepackage[hyphens]{url}
\usepackage{hyperref}
\hypersetup{
    breaklinks=true,
    colorlinks=true,
    linkcolor=blue,
    filecolor=blue,      
    urlcolor=blue,
    citecolor=blue
    }
\usepackage[hyphenbreaks]{xurl}
\usepackage{amsmath}
\usepackage{extpfeil}
\usepackage{amsthm}
\usepackage{amssymb}
\usepackage{booktabs}
\usepackage{adjustbox}
\usepackage{ltablex,array}
\usepackage{longtable}
\usepackage{makecell}
\usepackage{marvosym}
\newcommand*{\faktor}[2]{
  \raisebox{0.5\height}{\ensuremath{#1}}
  \mkern-5mu\diagup\mkern-4mu
  \raisebox{-0.5\height}{\ensuremath{#2}}
} 
\newcommand\restr[2]{{
\left.
\kern-
\nulldelimiterspace 
#1 
\right|_{#2} 
}}
\usepackage{tikz}
\usepackage{tikz-cd}
\tikzcdset{scale cd/.style={every label/.append style={scale=#1},
    cells={nodes={scale=#1}}}}
\usepackage{quiver}
\usetikzlibrary{positioning} \usetikzlibrary{decorations.pathreplacing}
\usetikzlibrary{arrows}
\tikzset{>=stealth}

\def\colim{\qopname\relax m{colim}}

\newtheorem{theorem}{Theorem}[section]
\newtheorem{proposition}[theorem]{Proposition}

\newtheorem{corollary}[theorem]{Corollary}
\newtheorem{question}[theorem]{Question}
\theoremstyle{definition}
\newtheorem{remark}[theorem]{Remark}
\newtheorem{definition}[theorem]{Definition}

\newtheorem{example}[theorem]{Example}
\theoremstyle{definition}
\newenvironment{usethmcounterof}[1]{%
  \renewcommand\thetheorem{\ref{#1}}\theorem}{\endtheorem\addtocounter{theorem}{-1}}

\begin{document}

\title{Positively closed $Sh(B)$-valued models}
\author{Kristóf Kanalas}

\date{}

\maketitle

\begin{abstract}
    We study positively closed and strongly positively closed topos-valued models of coherent theories. Positively closed is a global notion (it is defined in terms of all possible outgoing homomorphisms), while strongly positively closed is a local notion (it only concerns the definable sets inside the model). For $\mathbf{Set}$-valued models of coherent theories they coincide.
    
    We prove that if $\mathcal{E}=Sh(B,\tau _{coh})$ for a complete Boolean algebra, then positively closed but not strongly positively closed $\mathcal{E}$-valued models of coherent theories exist, yet, there is an alternative local property which characterizes positively closed $\mathcal{E}$-valued models.

    A large part of our discussion is given in the context of infinite quantifier geometric logic, dealing with the fragment $L^g_{\kappa \kappa }$ where $\kappa $ is weakly compact.
    \end{abstract}


\vspace{5mm}

\tableofcontents

\section{Introduction}

Models of coherent theories internal to Grothendieck toposes capture a wide range of classical mathematical concepts. For example, if $T$ is the theory of local rings and $X$ is a topological space, then a $T$-model internal to $Sh(X)$ is a locally ringed space with undelying space $X$. If $T$ is the theory of fields, then a $T$-model internal to the presheaf topos $\mathbf{Set}^{\circlearrowleft}$ (where $\circlearrowleft $ is a one-object category with a free endomorphism) is a difference field, see \cite{TOMASIC2022108328}. Finally, if $T$ is any theory and $B$ is a complete Boolean algebra, then a $T$-model in $Sh(B,\tau _{can})$ (that is: sheaves wrt.~the canonical topology, i.e.~covers are arbitrary unions) is the same as a $B$-valued model, see \cite[Chapter 4]{makkai}. These examples motivate the development of topos-valued positive model theory.

A key concept in positive model theory is that of a positively closed model; a model such that homomorphisms out of it are not just preserving but also reflecting the validity of positive existential formulas, see \cite{poizatyaacov}. This is a global notion, as it is defined in terms of all outgoing homomorphisms. In the $\mathbf{Set}$-valued case there is an equivalent local notion, called ``strongly positively closed'', defined in terms of the definable sets of our model. The goal of this paper is to prove that if $\mathcal{E}=Sh(X)$ for an extremally disconnected Stone space (equivalently: $\mathcal{E}=Sh(B,\tau _{coh})$ for a complete Boolean algebra, where $\tau _{coh}$ is the finite union topology), then for $\mathcal{E}$-valued models of coherent theories these two notions do not coincide anymore, however, there is an alternative local property which characterizes positively closed models. These results are new even for the coherent fragment $L_{\omega \omega }^g$, but we prove them for $L_{\kappa \kappa }^g$ where $\kappa $ is weakly compact. (Recall that $L^g_{\lambda \kappa }$ is the geometric fragment of $L_{\lambda \kappa }$; the set of formulas $\forall \vec{x}(\varphi (\vec{x})\to \psi (\vec{x}))$, where $\varphi $ and $\psi $ are positive existential. When $\lambda =\kappa =\omega $ such theories are called coherent or h-inductive, and they form the subject of positive model theory. Infinite quantifier geometric theories are inevitable, e.g.~in the study of abstract elementary classes, or more generally in the study of accessible categories.)

We refer to \cite[Definition 2.3]{bvalued} for the definition of a $(\lambda ,\kappa )$-coherent category. $(\kappa ,\kappa )$-coherent will be abbreviated as $\kappa $-coherent. A $\kappa $-topos is a Grothendieck topos which is $(\infty ,\kappa )$-coherent, see \cite{presheaftype} for an overview. 

We remark that $(\lambda ,\kappa )$-coherent categories are precisely the syntactic categories of theories in $L^g_{\lambda \kappa }$. When a theory $T\subseteq L_{\lambda \kappa }^g$ is replaced by its syntactic category $\mathcal{C}_T$, the internal models of $T$ in a $\kappa $-topos $\mathcal{E}$ get identified with $(\lambda ,\kappa )$-coherent functors $\mathcal{C}_T\to \mathcal{E}$, and the homomorphisms get identified with the natural transformations. This ``external'' approach allows a more direct use of category-theoretic techniques, which we will make use of.

Fix a weakly compact cardinal $\kappa $ (allowing $\kappa =\omega $). We will study $\kappa $-coherent functors $\mathcal{C}\to Sh(B,\tau _{\kappa -coh})$ where $\mathcal{C}$ is a $\kappa $-coherent category with $\kappa $-small disjoint coproducts, $B$ is a $\kappa $-coherent Boolean algebra and $\tau _{\kappa -coh}$ is formed by $\kappa $-small unions. We will adopt the terminology from \cite{bvalued} and call such functors $Sh(B)$-valued models. The setup is motivated by an infinitary version of Lurie's theorem: if $\kappa $ is weakly compact and $\mathcal{C}$ is $\kappa$-coherent with $\kappa $-small disjoint coproducts then any $\kappa $-regular functor $\mathcal{C}\to \mathbf{Set}$ factors uniquely as a $Sh(B)$-valued model followed by global sections (\cite[Theorem 4.6]{bvalued}).

Building on this theorem, we obtain the following results: In Section 4.~we describe a functor which takes a $\kappa $-lex map $F:\mathcal{C}\to \mathbf{Set}$ to a distributive lattice $L^1(F)$, the ``Lindenbaum-Tarski algebra of closed formulas with parameters from $F$''. Then we characterize strongly positively closed and positively closed models $M:\mathcal{C}\to Sh(B)$ in terms of the lattice $L^1(\Gamma M)$ (Theorems \ref{strongposclequiv} and \ref{posclequiv}). Since $L^1(\Gamma M)$ is computed directly from the model $M$, these are local properties. When $\kappa =\omega $ and $B$ is a complete Boolean algebra we obtain an even more explicit local characterization of positively closed models (Theorem \ref{Bcompletestrongposcl}). We give several explicit examples of positively closed but not strongly positively closed $Sh(B)$-valued models of coherent theories (Examples \ref{Open(Q)} and \ref{Closed(X)}).

Finally, we remark that positively closed and strongly positively closed topos-valued models have been studied in \cite{exclosedkamsma}. They consider $(\lambda ,\omega )$-coherent functors $\mathcal{C}\to \mathcal{E}$ where $\mathcal{C}$ is $(\lambda ,\omega )$-coherent (i.e.~$\lambda $-geometric) and $\mathcal{E}$ is an $\omega $-topos (i.e.~a Grothendieck topos). In logical terms: we investigate models of $L_{\kappa \kappa }^g$-theories in a special class of $\kappa $-toposes, for $\kappa $ weakly compact, while \cite{exclosedkamsma} is about models of $L_{\lambda \omega }^g$-theories in arbitrary toposes. Whenever our results are related to theirs, it will be noted.

\section{Type space functors}

Let $\mathcal{C}$ be a coherent category. We can think of its objects as positive existential formulas, arrows as provably functional positive existential formulas, and subobjects as implications between formulas sharing the same set of free variables. Indeed, if $\mathcal{C}$ is a syntactic category then subobjects can be represented by monos of the form
\[
[\varphi (\vec{x})]\xrightarrow{[\varphi (\vec{x})\wedge \vec{x}=\vec{x'}]}[\psi (\vec{x'})]
\]
where $T\vdash \varphi (\vec{x})\to \psi (\vec{x})$. Therefore we think of $Sub_{\mathcal{C}}:\mathcal{C}^{op}\to \mathbf{DLat}$ as the functor which sends an object $[\psi (\vec{x})]$ to the Lindenbaum-Tarski algebra of positive existential formulas below $\psi (\vec{x})$. What (positive) model theorists call a type is just a prime-filter on this Lindenbaum-Tarski algebra, so the composite
\[
S_{\mathcal{C}}:\mathcal{C}\xrightarrow{Sub_{\mathcal{C}}}\mathbf{DLat}^{op}\xrightarrow{Spec} \mathbf{Set} 
\]
(sending an object $x$ to the set of prime filters on its subobject lattice) is often referred as the type space functor of $\mathcal{C}$.

We think of $\mathcal{C}\to \mathbf{Set}$ coherent functors as models. A tuple of elements $\vec{a}$ in a model $M$ has a type, given by the set of formulas $\varphi (\vec{x})$ for which $M\models \varphi(\vec{a})$. In the categorical language this corresponds to a natural transformation $tp:M\Rightarrow S_{\mathcal{C}}$ given by $tp_{x}(a)=\{u\subseteq x: a\in Mu \}\in S_{\mathcal{C}}(x)$.

In model theory the notion of a type plays a key role. In this paper type space functors will be essential. The above definitions suffice for the $\mathbf{Set}$-valued model theory of coherent categories. Now we extend them, to study $\kappa $-topos-valued models of $\kappa $-coherent categories.

\begin{definition}
  Let $\kappa $ be a weakly compact cardinal (in this paper this allows $\kappa =\omega$). Take a $\kappa $-coherent category $\mathcal{C}$ and a $\kappa $-topos $Sh(\mathcal{D})=Sh(\mathcal{D},\tau )$. We define the $Sh(\mathcal{D})$-valued type space functor $S_{\mathcal{C}}=S_{\mathcal{C}}^{Sh(\mathcal{D})}$ of $\mathcal{C}$ as

\[\begin{tikzcd}
	{\mathcal{C}} &&& {Sh(\mathcal{D})} \\
	x & \mapsto && {d\mapsto \mathbf{Coh}_{\kappa }(Sub_{\mathcal{C}}(x),Sub_{Sh(\mathcal{D})}(\widehat{d}))}
	\arrow["{S_{\mathcal{C}}}", from=1-1, to=1-4]
\end{tikzcd}\]
\end{definition}

\begin{remark}
\label{subisomega}
    The presheaf $Sub_{Sh(\mathcal{D})}(\widehat{\bullet })$ is the subobject classifier $ \Omega $, in particular it does not depend on the site presentation and it is a sheaf.
\end{remark}

\begin{remark}
    We assume $\kappa $ to be weakly compact to ensure that $\kappa $-small epimorphic families form a $\kappa $-topology, see \cite[Remark 2.8]{presheaftype}. At this point this is not strictly necessary, but for most results it will be. Type space functors could also be defined more generally, on so-called $\kappa $-sites, for that approach we refer to the analogous section of \cite{thesis}.
\end{remark}

\begin{proposition}
    The functor $S_{\mathcal{C}}$ is well-defined, i.e.~$\mathbf{Coh}_{\kappa }(Sub_{\mathcal{C}}(x),Sub_{Sh(\mathcal{D})}(\widehat{\bullet }))$ is a sheaf (of posets) wrt.~$\tau $.
\end{proposition}

\begin{proof}
    Let $(h_i:d_i\to d)_i$ be a $\tau $-family. Write $d_i\xleftarrow{h_{ij,i}}d_{ij}\xrightarrow{h_{ij,j}}d_j$ for the pullback of $d_i\to d\leftarrow d_j$. By Remark \ref{subisomega}
    
\[
\adjustbox{width=\textwidth}{
\begin{tikzcd}
	&& {Sub_{Sh(\mathcal{D})}(\widehat{d})} \\
	{Sub_{Sh(\mathcal{D})}(\widehat{d_i})} && {Sub_{Sh(\mathcal{D})}(\widehat{d_j})} & {Sub_{Sh(\mathcal{D})}(\widehat{d_k})} & \dots \\
	& {Sub_{Sh(\mathcal{D})}(\widehat{d_{ij}})} && \dots
	\arrow["{\widehat{h_i}^{-1}}"', from=1-3, to=2-1]
	\arrow["{\widehat{h_j}^{-1}}"', from=1-3, to=2-3]
	\arrow["{\widehat{h_k}^{-1}}", from=1-3, to=2-4]
	\arrow["{\widehat{h_{ij,i}}^{-1}}"', from=2-1, to=3-2]
	\arrow["{\widehat{h_{ij,j}}^{-1}}", from=2-3, to=3-2]
\end{tikzcd}
}
\]
    is a limit cone in $\mathbf{Set}$. Finally, given a function $F:Sub_{\mathcal{C}}(x)\to Sub_{Sh(\mathcal{D})}(\widehat{d})$, it preserves $\kappa $-small meets and $\kappa $-small joins iff each post-composition $\widehat{h_i}^{-1}\circ F$ does (as those are jointly monic). So this is a limit cone among $\kappa $-coherent lattices, which is preserved by $\mathbf{Coh}_{\kappa }(Sub_{\mathcal{C}}(x),-)$ (as a functor to $\mathbf{Set}$). So $S_{\mathcal{C}}(x)$ is a sheaf of sets, but also a sheaf of posets (with pointwise $\leq $) as post-composition preserves the ordering.
\end{proof}

\begin{example}
    Let $\mathcal{C}$ be a coherent category and let $X$ be a topological space. Then $S_{\mathcal{C}}:\mathcal{C}\to Sh(X)$ is given by $x\mapsto [U\mapsto Cont(U,S_{\mathcal{C}}(x))]$ where $S_{\mathcal{C}}(x)$ is the spectral space of prime filters on $Sub_{\mathcal{C}}(x)$. Indeed, we have a natural isomorphism $Cont(U,S_{\mathcal{C}}(x))\cong \mathbf{Coh}(Sub_{\mathcal{C}}(x),\mathcal{O}(U))$ (of posets): given a continuous map $f$ the corresponding homomorphism is $f^{-1}$, and given a homomorphism $h$ the corresponding function is $U\ni p \mapsto \{u\hookrightarrow x: p\in h(u) \}$.
\end{example}

\begin{definition}
\label{deftp}
    Let $\kappa $ be weakly compact, $\mathcal{C}$ a $\kappa $-coherent category, and $Sh(\mathcal{D})$ a $\kappa $-topos. Given a $\kappa $-coherent functor $M:\mathcal{C}\to Sh(\mathcal{D})$ we define a natural transformation $tp_M:M\Rightarrow S_{\mathcal{C}}$ by:
\[
\adjustbox{width=\textwidth}{
\begin{tikzcd}
	{Mx(d)} & { s} \\
	{\mathbf{Coh}_{\kappa}(Sub_{\mathcal{C}}(x), Sub_{Sh(\mathcal{D})}(\widehat{d}) )} & {x\supseteq u\ \mapsto \ \bigcup  \{im(\widehat{d'}\xrightarrow{\widehat{h}} \widehat{d}) : \restr{s}{d'}\in Mu(d') \}}
	\arrow[from=1-1, to=2-1]
	\arrow[shorten <=3pt, shorten >=3pt, maps to, from=1-2, to=2-2]
\end{tikzcd}
}
\]
In other terms: $tp_{M,x,d}(s)$ takes $u\hookrightarrow x$ to the pullback of $Mu\hookrightarrow Mx$ along $s:\widehat{d}\to Mx$.
\end{definition}

\begin{remark}
    Alternatively, $(tp_M)_{x,d}$ can be described as $s\mapsto [u\mapsto \chi _{Mu,d}(s)]$ where $\chi _{Mu}:Mx\to \Omega $ is the characteristic map of $Mu\hookrightarrow Mx$.
\end{remark}

\begin{proposition}
    $tp_M$ is well-defined.
\end{proposition}

\begin{proof}
    $tp_{M,x,d}(s): Sub_{\mathcal{C}}(x)\xrightarrow{M}Sub_{Sh(\mathcal{D})}(Mx)\xrightarrow{s^{-1}} Sub_{Sh(\mathcal{D})}(\widehat{d})$ preserves $\kappa $-small meets and $\kappa $-small joins: clear.

    $tp_{M,x}$ is a natural transformation: fix $k:d_0\to d$. The square 
\[\begin{tikzcd}
	{Mx(d)} && {Mx(d_0)} \\
	\\
	{\mathbf{Coh}_{\kappa }(Sub_{\mathcal{C}}(x), Sub_{Sh(\mathcal{D})}(\widehat{d}) )} && {\mathbf{Coh}_{\kappa }(Sub_{\mathcal{C}}(x), Sub_{Sh(\mathcal{D})}(\widehat{d_{0}}) )}
	\arrow["{Mx(k)}", from=1-1, to=1-3]
	\arrow[from=1-1, to=3-1]
	\arrow[from=1-3, to=3-3]
	\arrow["{(k^{-1})_{\circ }}"', from=3-1, to=3-3]
\end{tikzcd}\]
commutes because for any $s\in Mx(d)$ we have $k^{-1}\circ s^{-1}\circ M = (sk)^{-1}\circ M$.

    $tp_M$ is a natural transformation: fix $f:x\to y$ and $d\in \mathcal{D}$. The square 
\[\begin{tikzcd}
	{Mx(d)} && {My(d)} \\
	\\
	{\mathbf{Coh}_{\kappa }(Sub_{\mathcal{C}}(x), Sub_{Sh(\mathcal{D})}(\widehat{d}) )} && {\mathbf{Coh}_{\kappa }(Sub_{\mathcal{C}}(y), Sub_{Sh(\mathcal{D})}(\widehat{d}) )}
	\arrow["{Mf_d}", from=1-1, to=1-3]
	\arrow[from=1-1, to=3-1]
	\arrow[from=1-3, to=3-3]
	\arrow["{(f^{-1})^{\circ }}"', from=3-1, to=3-3]
\end{tikzcd}\]
commutes by the commutativity of 
\[\begin{tikzcd}
	{Sub_{\mathcal{C}}(y)} && {Sub_{Sh(\mathcal{D})}(My)} \\
	{Sub_{\mathcal{C}}(x)} && {Sub_{Sh(\mathcal{D})}(Mx)} && {Sub_{Sh(\mathcal{D})}(\widehat{d})}
	\arrow["{M}", from=1-1, to=1-3]
	\arrow["{f^{-1}}"', from=1-1, to=2-1]
	\arrow["{Mf^{-1}}", from=1-3, to=2-3]
	\arrow["{M}"', from=2-1, to=2-3]
	\arrow["{s^{-1}}", from=2-3, to=2-5]
\end{tikzcd}\]
\end{proof}

Now we prove that $tp$ is the smallest natural transformation from a model to the type space functor:

\begin{proposition}
\label{tpisminimal}
     Let $\kappa $ be weakly compact, $\mathcal{C}$ a $\kappa $-coherent category, $Sh(\mathcal{D})$ a $\kappa $-topos, and $M:\mathcal{C}\to Sh(\mathcal{D})$ a $\kappa $-coherent functor. Given a natural transformation $\beta :M\Rightarrow S_{\mathcal{C}}$ we have $\beta _{x,d}(s)(u)\supseteq tp_{M,x,d}(s)(u)$ for any $x\in \mathcal{C}$, $d\in \mathcal{D}$, $s\in Mx(d)$ and $u\hookrightarrow \widehat{x}$.
\end{proposition}

\begin{proof}
    We have to prove that if for some $h:d'\to d$ we have $\restr{s}{d'}\in Mu(d')$ then $im(\widehat{d'}\xrightarrow{\widehat{h}}\widehat{d})\subseteq \beta _{x,d}(s)(u)$. By naturality there is a commutative square

\[\begin{tikzcd}
	{Sub_{\mathcal{C}}(x)} && {Sub_{Sh(\mathcal{D})}(\widehat{d})} \\
	\\
	{Sub_{\mathcal{C}}(u)} && {Sub_{Sh(\mathcal{D})}(\widehat{d'})}
	\arrow["{\beta_{u,d'}(\restr{s}{d'})}", from=3-1, to=3-3]
	\arrow[from=1-1, to=3-1]
	\arrow["{\beta _{x,d}(s)}", from=1-1, to=1-3]
	\arrow[from=1-3, to=3-3]
\end{tikzcd}\]
As the lower composite takes $u\hookrightarrow x$ to $\top=\widehat{d'}\hookrightarrow \widehat{d'}$ so does the upper one, meaning that the pullback of $\beta _{x,d}(s)(u)\hookrightarrow \widehat{d}$ along $\widehat{d'}\to \widehat{d}$ is the maximal subobject.
\end{proof}

\begin{definition}
\label{elementarynattr}
    Let $\mathcal{C}$, $\mathcal{D}$ be categories with finite limits and let $F,G:\mathcal{C}\to \mathcal{D}$ be lex functors. A natural transformation $\alpha :F\Rightarrow G$ is \emph{elementary} if the naturality squares at monomorphisms are pullbacks.
\end{definition}

\begin{remark}
    Think of $\mathcal{C}$ as a syntactic category, and $M,N:\mathcal{C}\to \mathbf{Set}$ as models. Then a map $M\Rightarrow N$ is elementary iff the squares
    
\[
\adjustbox{scale=1}{
\begin{tikzcd}
	{[\varphi (\vec{x})]^M} && {[\vec{x}=\vec{x}]^M} \\
	\\
	{[\varphi (\vec{x})]^N} && {[\vec{x}=\vec{x}]^N}
	\arrow[hook, from=1-1, to=1-3]
	\arrow[from=1-1, to=3-1]
	\arrow[from=1-3, to=3-3]
	\arrow[hook, from=3-1, to=3-3]
\end{tikzcd}
}
\]
    are pullbacks. So maps between models are functions preserving the formulas in the fragment (commutativity), elementary maps are functions preserving and reflecting them (pullback). When $\mathcal{C}$ is a coherent category (i.e.~formulas in $\mathcal{C}$ are positive existential), maps are just homomorphisms, elementary maps are those homomorphisms which reflect positive existential formulas. These are called immersions in positive logic. 
\end{remark}

\begin{remark}
    Elementary maps were first defined in \cite{barr}. They are also called elementary in \cite{lurie}. They are called immersions e.g.~in \cite{exclosedkamsma}. They are called mono-cartesian in \cite{GARNER2020102831}.
\end{remark}

\begin{proposition}
\label{elemthenmonic}
    Let $\mathcal{C}$, $\mathcal{D}$ be categories with finite limits, and let $F,G:\mathcal{C}\to \mathcal{D}$ be lex functors. If $\alpha :F\Rightarrow G$ is elementary, then each component $\alpha _x$ is monic.
\end{proposition}

\begin{proof}
    \cite[Proposition 5.8]{bvalued}.
\end{proof}

\begin{proposition}
\label{elemifftpcommutes}
    Let $\kappa $ be weakly compact, $\mathcal{C}$ a $\kappa $-coherent category, $Sh(\mathcal{D})$ a $\kappa $-topos, and $M,N:\mathcal{C}\to Sh(\mathcal{D})$ $\kappa $-coherent functors. Given a natural transformation $\alpha :M\Rightarrow N$ it is elementary iff $tp_N\circ \alpha = tp_M$.
\end{proposition}

\begin{proof}
    \fbox{$\Rightarrow $} We need that for any $s\in Mx(d)$, the homomorphism $tp_{M,x,d}(s):Sub_{\mathcal{C}}(x)\xrightarrow{M} Sub_{Sh(\mathcal{D})}(Mx)\xrightarrow{s^{-1}}Sub_{Sh(\mathcal{D})}(\widehat{d})$ equals $(tp_N\circ \alpha )_{x,d}(s):Sub_{\mathcal{C}}(x)\xrightarrow{N} Sub_{Sh(\mathcal{D})}(Nx)\xrightarrow{\alpha _x^{-1}}Sub_{Sh(\mathcal{D})}(Mx) \xrightarrow{s^{-1}}Sub_{Sh(\mathcal{D})}(\widehat{d})$. But this is the case, since $\alpha $ is elementary iff 
\[
\adjustbox{scale=1}{
\begin{tikzcd}
	&& {Sub_{Sh(\mathcal{D})}(Mx)} \\
	{Sub_{\mathcal{C}}(x)} \\
	&& {Sub_{Sh(\mathcal{D})}(Nx)}
	\arrow["M", from=2-1, to=1-3]
	\arrow["N"', from=2-1, to=3-3]
	\arrow["{\alpha _x^{-1}}"', from=3-3, to=1-3]
\end{tikzcd}
}
\]
    commutes.

    \fbox{$\Leftarrow $} Assume that 
\[\begin{tikzcd}
	{Mu(d)} && {Mx(d)} \\
	{Nu(d)} && {Nx(d)}
	\arrow[hook, from=1-1, to=1-3]
	\arrow["{\alpha _{u,d}}"', from=1-1, to=2-1]
	\arrow["{\alpha _{x,d}}", from=1-3, to=2-3]
	\arrow[hook, from=2-1, to=2-3]
\end{tikzcd}\]
is not a pullback for some $u\hookrightarrow x$ in $\mathcal{C}$ and $d\in \mathcal{D}$. That is, we can find $s\in Mx(d)$ s.t.~$\alpha _{x,d}(s)\in Nu(d)$ but $s\not \in Mu(d)$. Then $s^{-1}\circ \alpha _x^{-1}\circ N$ takes $u\hookrightarrow x$ to the maximal subobject $\widehat{d}\to \widehat{d}$ but $s^{-1}\circ M$ does not.
\end{proof}

\section{Maps to the type space functor}

In \cite[Section 3]{bvalued} we described the method of diagrams in the language of categorical logic, based on the following idea: adding constant symbols to a theory is the same as taking a (sufficiently filtered) colimit of slices of the syntactic category. Indeed, moving from $\mathcal{C}$ to a slice $\mathcal{C}\to \faktor{\mathcal{C}}{x}$ adds a ``generic'' global element to $x$, namely the diagonal $x\xrightarrow{\Delta } x\times x$. What properties these added elements will have, depends on the indexing diagram of the colimit. Usually $\mathcal{C}$ sits at an initial position in the diagram and by writing $\varphi :\mathcal{C}\to \mathcal{C}'$ for the cocone map we obtain an extension, s.t.~precomposition with $\varphi $ yields an equivalence between $\mathcal{C}'$-models and $\mathcal{C}$-models equipped with a correct interpretation of the constants. (If $\kappa $ is weakly compact then a $\kappa $-filtered colimit of $\kappa $-coherent categories is $\kappa $-coherent, see \cite[Theorem 4.1.i)]{bvalued}. When $\kappa $ is an arbitrary regular cardinal we have to treat $\mathcal{C}'$ as a $\kappa $-site, but this technicality will not play a role in this paper.)

In \cite{bvalued} we used this idea to name the elements of a $\kappa $-lex functor $K:\mathcal{C}\to \mathcal{D}$, i.e.~to build an extension $\mathcal{C}_K$ of $\mathcal{C}$ whose $\mathcal{D}$-models correspond to the continuations of $K$. Here our goal is to prove that a natural transformation $\beta :M\Rightarrow S_{\mathcal{C}}$ from a model to the type space functor can be factored through a model $N$ as $M\xRightarrow{\alpha }N\xRightarrow{tp_N}S_{\mathcal{C}}$, at least in a lax sense: that $\beta \subseteq tp_N\circ \alpha $ (pointwise). This will be a key technical tool in the study of positively closed models.

Before stating the precise statement, we discuss whether topos-valued types can be realized by topos-valued models.

\begin{definition}
\label{realtypes}
    Let $\kappa $ be weakly compact, and $\mathcal{E}$ a $\kappa $-topos. We say that $\mathcal{E}$ \emph{can realize $\kappa $-types} if given a $\kappa$-coherent category $\mathcal{C}$ and a 
$\kappa$-coherent functor $p:Sub_{\mathcal{C}}(1)\to \mathcal{E}$, there exists a $\kappa$-coherent functor $M:\mathcal{C}\to \mathcal{E}$ with $p\leq \restr{M}{Sub_{\mathcal{C}}(1)}$: 
\[\begin{tikzcd}
	{Sub_{\mathcal{C}}(1)} && {\mathcal{E}} \\
	\\
	{\mathcal{C}}
	\arrow[""{name=0, anchor=center, inner sep=0}, "p", from=1-1, to=1-3]
	\arrow[""{name=0p, anchor=center, inner sep=0}, phantom, from=1-1, to=1-3, start anchor=center, end anchor=center]
	\arrow[hook, from=1-1, to=3-1]
	\arrow[""{name=1, anchor=center, inner sep=0}, "M"{description}, dashed, from=3-1, to=1-3]
	\arrow[""{name=1p, anchor=center, inner sep=0}, phantom, from=3-1, to=1-3, start anchor=center, end anchor=center]
	\arrow["\leq"{marking, allow upside down}, shift right=5, draw=none, from=0p, to=1p]
\end{tikzcd}\]
\end{definition}

\begin{example}
    Take $\kappa =\omega $ and $\mathcal{E}=Sh(X)$ where $X$ is a topological space. We can think of the coherent category $\mathcal{C}$ as the syntactic category of some coherent theory $T\subseteq L_{\omega \omega }^g$, then a coherent functor $M:\mathcal{C}\to Sh(X)$ can be identified with a $Sh(X)$-valued (or internal) model of $T$. A $Sh(X)$-valued model can be described explicitly, as a sheaf of $L$-structures, such that each stalk is a model (see \cite[Remark 3, Lecture 16X]{lurie}). The above lifting property now says, that if we associate to each closed positive existential formula an open set in a compatible way, then we can find a $Sh(X)$-valued model, such that at any point $a\in X$, the stalk $M_a$ satisfies all prescribed formulas (i.e.~the ones whose associated open set contains $a$).
\end{example}

\begin{remark}
    In Definition \ref{realtypes} what we call ``type'' should be called a (positive) 0-type (without parameters), as it is a (generalized) prime filter on the Lindenbaum-Tarski algebra of closed positive existential formulas. However, realizing 0-types suffices to realize n-types. Let $x$ be an object of $\mathcal{C}$ (e.g.~the object $[\vec{x}=\vec{x}]$ of a syntactic category). Given a coherent functor $Sub_{\mathcal{C}}(x)\to \mathcal{E}$, we can apply the lifting property with $\faktor{\mathcal{C}}{x}$ to obtain
\[\begin{tikzcd}
	{Sub_{\faktor{\mathcal{C}}{x}}(1)} & {Sub_{\mathcal{C}}(x)} && {\mathcal{E}} \\
	{\faktor{\mathcal{C}}{x}} & {}
	\arrow[equals, from=1-1, to=1-2]
	\arrow[hook', from=1-1, to=2-1]
	\arrow["p", from=1-2, to=1-4]
	\arrow[between={0}{0.4}, Rightarrow, from=1-2, to=2-2]
	\arrow["M"', dashed, from=2-1, to=1-4]
\end{tikzcd}\]
Write $M_0$ for the composite $\mathcal{C}\to \faktor{\mathcal{C}}{x}\xrightarrow{M}\mathcal{E}$ and take $i:u\hookrightarrow x$. As $M$ preserves pullbacks we get
\[\begin{tikzcd}
	&&&& {p(i)} \\
	u && x && {M(i)} && 1 \\
	{u\times x} && {x\times x} && {M_0(u)} && {M_0(x)} \\
	& x
	\arrow[dashed, from=1-5, to=2-5]
	\arrow[hook, from=1-5, to=2-7]
	\arrow[""{name=0, anchor=center, inner sep=0}, "i", hook, from=2-1, to=2-3]
	\arrow[""{name=0p, anchor=center, inner sep=0}, phantom, from=2-1, to=2-3, start anchor=center, end anchor=center]
	\arrow["{\langle 1 ,i\rangle}"', from=2-1, to=3-1]
	\arrow["i"{description, pos=0.7}, hook, from=2-1, to=4-2]
	\arrow["{\Delta }", from=2-3, to=3-3]
	\arrow["1"{description, pos=0.7}, from=2-3, to=4-2]
	\arrow[""{name=1, anchor=center, inner sep=0}, hook, from=2-5, to=2-7]
	\arrow[from=2-5, to=3-5]
	\arrow["s", from=2-7, to=3-7]
	\arrow[""{name=2, anchor=center, inner sep=0}, "{i\times 1}", from=3-1, to=3-3]
	\arrow[""{name=2p, anchor=center, inner sep=0}, phantom, from=3-1, to=3-3, start anchor=center, end anchor=center]
	\arrow["{\pi _2}"', from=3-1, to=4-2]
	\arrow["{\pi _2}", from=3-3, to=4-2]
	\arrow[""{name=3, anchor=center, inner sep=0}, "{M_0(i)}"', hook, from=3-5, to=3-7]
	\arrow["pb"{description, pos=0.4}, draw=none, from=0p, to=2p]
	\arrow["pb"{description}, draw=none, from=1, to=3]
\end{tikzcd}\]
So the $M$-image of $\Delta $ gives a global element $s:1\to M_0(x)$ whose restriction to $p(i)\subseteq 1$ factors through the subobject $M_0(i):M_0(u)\hookrightarrow M_0(x)$.

In the situation of the previous example this means, that we get a $Sh(X)$-valued model $M_0$, together with a global section $s$ of $M_0(x)$, which lies in the subsheaf $M_0(u)$ over the open set which was prescribed for $i:u\hookrightarrow x$. In other terms, when $\mathcal{C}$ is a syntactic category and $x=[x_1=x_1\wedge \dots \wedge x_n=x_n]$, then at any point $a\in X$ the germ of the global section $s_a\in M_a^n$ (or $s_a\in (Mx_1)_a\times \dots (Mx_n)_a$ if there are different sorts) satisfies all formulas $\varphi (\vec{x})$ whose prescribed open contains $a$.  
\end{remark}

\begin{theorem}
\label{setrealtypes}
    $\mathbf{Set}$ can realize $\kappa $-types for any strongly compact $\kappa $. More is true: any $\kappa$-coherent functor $p:Sub_{\mathcal{C}}(1)\to \mathbf{Set}$ admits a strict extension to $\mathcal{C}$ (meaning: the triangle commutes).
\end{theorem}

\begin{proof}
    There is a conservative $\kappa $-coherent functor $\langle M_i \rangle _i : \mathcal{C}\to \mathbf{Set}^{I}$, hence there is a conservative $\kappa $-coherent functor $J:\mathcal{C}\to \mathcal{D}$ where $\mathcal{D}$ is a small Boolean $\kappa $-coherent category and $Sub_{\mathcal{D}}(1)$ is a powerset Boolean algebra.

    The $\kappa $-complete prime filter and its complement can be separated by a $\kappa $-complete ultrafilter $U:Sub_{\mathcal{D}}(1)\to 2$. It suffices to prove that $U$ extends to a coherent functor $M_U:\mathcal{D}\to \mathbf{Set}$. 

\[\begin{tikzcd}
	{Sub_{\mathcal{C}}(1)} && {\mathbf{Set}} \\
	{Sub_{\mathcal{D}}(1)} \\
	{\mathcal{D}}
	\arrow["p", from=1-1, to=1-3]
	\arrow[from=1-1, to=2-1]
	\arrow["U"{description}, from=2-1, to=1-3]
	\arrow[from=2-1, to=3-1]
	\arrow["{M_U}"', dashed, from=3-1, to=1-3]
\end{tikzcd}\]
    
    Write $\mathcal{D}_U=\colim _{v\in U}\faktor{\mathcal{D}}{v}$, then this is a consistent ($0\neq 1$) $\kappa$-coherent category and $\mathcal{D}\to \mathcal{D}_U$ is $\kappa$-coherent (see \cite[Theorem 4.1.i)]{bvalued}). So there is a $\kappa$-coherent functor $\mathcal{D}_U\to \mathbf{Set}$. It follows that $M_U:\mathcal{D}\to \mathcal{D}_U\to \mathbf{Set}$ takes each $v\in U$ to $1\in \mathbf{Set}$. Since $U$ is maximal, this proves $\restr{M_U}{Sub_{\mathcal{D}}(1)}=U$.
\end{proof}

\begin{theorem}
\label{Bvaltypesreal}
    Take a complete Boolean algebra $B$. Then $Sh(B,\tau _{\omega -coh})$ can realize $\omega $-types.
\end{theorem}

\begin{proof}
    By \cite[Theorem 3.2]{injdlat} complete Boolean algebras are injective objects (even in the category of distributive lattices), so there is a split monomorphism $B\hookrightarrow 2^I$. So by

\[\begin{tikzcd}
	{Sub_{\mathcal{C}}(1)} & {Sh(B,\tau _{coh})} & {Sh(2^I,\tau _{coh})} & {Sh(B,\tau _{coh})} \\
	\\
	{\mathcal{C}}
	\arrow["p", from=1-1, to=1-2]
	\arrow[hook, from=1-1, to=3-1]
	\arrow[from=1-2, to=1-3]
	\arrow[curve={height=-18pt}, equals, from=1-2, to=1-4]
	\arrow[from=1-3, to=1-4]
	\arrow[""{name=0, anchor=center, inner sep=0}, dashed, from=3-1, to=1-3]
	\arrow["\leq"{marking, allow upside down}, draw=none, from=1-2, to=0]
\end{tikzcd}\]
it is enough to prove that $Sh(2^I,\tau _{coh})$ can realize $\omega $-types. 

For any $i\in I$ fix an extension
\[\begin{tikzcd}
	{Sub_{\mathcal{C}}(1)} && {Sh(2^I,\tau _{coh})} && {\mathbf{Set}} \\
	\\
	{\mathcal{C}}
	\arrow[""{name=0, anchor=center, inner sep=0}, "p", from=1-1, to=1-3]
	\arrow[hook, from=1-1, to=3-1]
	\arrow["{\pi _i^*}", from=1-3, to=1-5]
	\arrow[""{name=1, anchor=center, inner sep=0}, "{M_i}"', from=3-1, to=1-5]
	\arrow["{=}"', draw=none, from=0, to=1]
\end{tikzcd}\]
where $\pi _i:2^I\to 2$ is the $i^{\text{th}}$ projection, in other terms it is the principal ultrafilter on $i$.

We define a functor $M:\mathcal{C}\to Sh(2^I,\tau _{coh})$ by $Mx(J)=\prod _{j\in J}M_jx$ (for $x\in \mathcal{C}$ and $J\subseteq I)$. This is clearly a functor $\mathcal{C}\to \mathbf{Set}^{(2^I)^{op}}$.

$Mx$ is a sheaf (even wrt.~$\tau _{can}$):
\[\begin{tikzcd}
	& {\prod _{\bigcup J_k} M_jx} \\
	{\prod _{J_1} M_jx} && {\prod _{J_2} M_jx} && \dots \\
	& {\prod _{J_1 \cap J_2} M_jx} && {\dots }
	\arrow[from=1-2, to=2-1]
	\arrow[from=1-2, to=2-3]
	\arrow[from=1-2, to=2-5]
	\arrow[from=2-1, to=3-2]
	\arrow[from=2-1, to=3-4]
	\arrow[from=2-3, to=3-2]
	\arrow[from=2-5, to=3-4]
\end{tikzcd}\]
is a limit.

$M$ is coherent: We need that the stalk at any point $U\in Spec(2^I)$ is a model. But the stalk is the ultraproduct $\faktor{\prod _I M_i}{U}$ which is a model.

Alternatively, one could also check this by hand: finite limits and effective epis are preserved (even by $\mathcal{C}\to \mathbf{Set}^{(2^I)^{op}}$) as the product of limit diagrams/ surjections is a limit diagram/ surjection. The initial object is taken to 
\[J\mapsto \begin{cases}
    * \text{ if } J=\emptyset \\
    \emptyset \text{ otherwise}
\end{cases}\]
which is the initial object in $Sh(2^I,\tau _{coh})$. Finally, binary unions are preserved: given $\vec{a}\in \prod _JM_ju \cup \prod _J M_jv \subseteq \prod _J M_jx$ we can find a cover $J=J_1\cup J_2$ such that the restriction of $\vec{a}$ to $J_1$ lives in $\prod _{J_1}M_ju$ and the restriction to $J_2$ lives in $\prod _{J_2}M_jv$. Just take $J_1=\{j:a_j\in M_{j}u \}$ and $J_2=\{j:a_j\in M_{j}v \}$.

$\restr{M}{Sub_{\mathcal{C}}(1)}\geq p$: Given $u\hookrightarrow 1$ and $i\in I$ we have $\{i\}\in p(u)\in Id(2^I)$ iff $M_i(u)=*$ and given $J\subseteq I$ we have $J\in M(u)\in Id(2^I)$ iff $\prod _JM_j(u) =*$ iff $\forall j\in J: M_j(u)=*$. Therefore $J\in p(u) \Rightarrow \forall j\in J:\{j\} \in p(u)\Rightarrow \forall j\in J:M_j(u)=* \Rightarrow J\in M(u)$.

\end{proof}

\begin{question}
Can we characterize those Grothendieck toposes which realize $\omega $-types? Are there any counterexamples?
\end{question}

Now we can state the main theorem of this section. The proof is moved to the Appendix, here we will give an outline.

\begin{theorem}
\label{factorbeta}
    Let $\kappa $ be weakly compact, $\mathcal{C}$ a $\kappa $-coherent category with $\kappa $-small disjoint coproducts, $B$ a $\kappa $-coherent Boolean algebra, $M:\mathcal{C}\to Sh(B,\tau _{\kappa -coh})$ a $\kappa $-coherent functor and $\beta :M\Rightarrow S_{\mathcal{C}}^{Sh(B)}$ a natural transformation. Assume moreover, that $Sh(B,\tau _{\kappa -coh})$ can realize $\kappa $-types. 

    Then there is a $\kappa $-coherent functor $N:\mathcal{C}\to Sh(B,\tau _{\kappa -coh})$ together with a natural transformation $\alpha :M\Rightarrow N$, such that $\beta \leq tp_N\circ \alpha  $ (pointwise).

\[\begin{tikzcd}
	{\mathcal{C}} &&& {} & {Sh(B)}
	\arrow[""{name=0, anchor=center, inner sep=0}, "M"{description}, curve={height=-30pt}, from=1-1, to=1-5]
	\arrow[""{name=0p, anchor=center, inner sep=0}, phantom, from=1-1, to=1-5, start anchor=center, end anchor=center, curve={height=-30pt}]
	\arrow[""{name=1, anchor=center, inner sep=0}, "{S_{\mathcal{C}}^{Sh(B)}}"{description}, curve={height=60pt}, from=1-1, to=1-5]
	\arrow[""{name=1p, anchor=center, inner sep=0}, phantom, from=1-1, to=1-5, start anchor=center, end anchor=center, curve={height=30pt}]
	\arrow["N"{description, pos=0.8}, curve={height=6pt}, from=1-1, to=1-5]
	\arrow["{\alpha }", between={0.3}{1}, Rightarrow, from=0, to=1-4]
	\arrow[""{name=2, anchor=center, inner sep=0}, "{\beta }"'{pos=0.4}, between={0.2}{1}, Rightarrow, from=0p, to=1p]
	\arrow["tp", between={0.1}{0.7}, Rightarrow, from=1-4, to=1]
	\arrow["\leq"{description}, draw=none, from=2, to=1-4]
\end{tikzcd}\]
\end{theorem}

\begin{proof}[Proof outline.]
The construction has the following steps:
\begin{itemize}
    \item Using the ``method of diagrams'' from \cite{bvalued}, we build a $\kappa $-coherent category $\mathcal{C}_{M,\beta }$, together with $\kappa $-coherent functors $\mathcal{C}\xrightarrow{\varphi }\mathcal{C}_{M,\beta }\xleftarrow{\psi }B$, and a natural transformation $\delta $  from $\mathcal{C}\xrightarrow{M}Sh(B)\xrightarrow{\psi ^*}Sh(\mathcal{C}_{M,\beta })$ to $\mathcal{C}\xrightarrow{\varphi }\mathcal{C}_{M,\beta }\xrightarrow{Y}Sh(\mathcal{C}_{M,\beta })$.
\item We construct lifts
\[\begin{tikzcd}
	B && {Sh(B)} \\
	{Sub_{\mathcal{C}_{M,\beta }}} \\
	{\mathcal{C}_{M,\beta }} & {}
	\arrow[""{name=0, anchor=center, inner sep=0}, "Y", from=1-1, to=1-3]
	\arrow[""{name=0p, anchor=center, inner sep=0}, phantom, from=1-1, to=1-3, start anchor=center, end anchor=center]
	\arrow[hook, from=1-1, to=2-1]
	\arrow["{\psi }"', curve={height=40pt}, from=1-1, to=3-1]
	\arrow[""{name=1, anchor=center, inner sep=0}, "{\chi _0}"{description}, from=2-1, to=1-3]
	\arrow[""{name=1p, anchor=center, inner sep=0}, phantom, from=2-1, to=1-3, start anchor=center, end anchor=center]
	\arrow[""{name=1p, anchor=center, inner sep=0}, phantom, from=2-1, to=1-3, start anchor=center, end anchor=center]
	\arrow[hook, from=2-1, to=3-1]
	\arrow["{\chi }"', curve={height=12pt}, from=3-1, to=1-3]
	\arrow["{=}"{description}, draw=none, from=0p, to=1p]
	\arrow[between={0.2}{0.5}, Rightarrow, from=1p, to=3-2]
\end{tikzcd}\]
and call the resulting natural transformation $\mu :Y\Rightarrow \chi  \psi $ (there is at most one such natural transformation, but we give it a name). The construction of $\chi _0$ is explicit, then we obtain $\chi $ from the assumption that $Sh(B)$ can realize types.
\item Finally, we construct a natural transformation $\nu $ and put everything together in 
\[\begin{tikzcd}[cramped]
	&&& {Sh(B)} \\
	\\
	{\mathcal{C}} &&&& {Sh(\mathcal{C}_{M,\beta })} && {Sh(B)} \\
	&&&& {}
	\arrow["{\psi ^*}"{description}, from=1-4, to=3-5]
	\arrow[""{name=0, anchor=center, inner sep=0}, curve={height=-18pt}, equals, from=1-4, to=3-7]
	\arrow["M"{description}, curve={height=-18pt}, from=3-1, to=1-4]
	\arrow[""{name=1, anchor=center, inner sep=0}, "{Y\varphi }"{description}, curve={height=-25pt}, from=3-1, to=3-5]
	\arrow[""{name=2, anchor=center, inner sep=0}, "{S_{\mathcal{C}}^{Sh(\mathcal{C}_{M,\beta })}}"{description}, curve={height=25pt}, from=3-1, to=3-5]
	\arrow["{S_{\mathcal{C}}^{Sh(B)}}"{description}, curve={height=70pt}, from=3-1, to=3-7]
	\arrow[""{name=3, anchor=center, inner sep=0}, "{\chi ^*}"{description}, from=3-5, to=3-7]
	\arrow["{\nu }", Rightarrow, from=3-5, to=4-5]
	\arrow["\delta"', curve={height=6pt}, between={0.1}{0.9}, Rightarrow, from=1-4, to=1]
	\arrow["{\widetilde{\mu }}"', between={0.2}{0.8}, Rightarrow, from=0, to=3]
	\arrow["{tp_{Y\varphi }}"', between={0.2}{0.8}, Rightarrow, from=1, to=2]
\end{tikzcd}\]
Then $N=\chi ^*Y\varphi:\mathcal{C}\to Sh(B)$ is a $\kappa $-coherent functor, and $\alpha =\chi ^*\delta \circ \widetilde{\mu }M:M\Rightarrow N$ is a natural transformation. We claim that $\nu \circ \chi ^*tp_{Y\varphi }=tp_N$ and that $\beta \leq tp_N\circ \alpha $. 
\end{itemize}
\end{proof}

\begin{corollary}
      Let $\mathcal{C}$ be coherent with finite disjoint coproducts, $B$ a complete Boolean algebra, $M:\mathcal{C}\to Sh(B,\tau _{\omega -coh})$ a coherent functor, and $\beta :M\Rightarrow S_{\mathcal{C}}=S_{\mathcal{C}}^{Sh(B)}$ a natural transformation. 
    
    Then there is a coherent functor $N:\mathcal{C}\to Sh(B)$ and a natural transformation $\alpha :M\Rightarrow N$ with $\beta \leq tp_N\circ \alpha $.
\end{corollary}

\begin{proof}
    By Theorem \ref{factorbeta} and Theorem \ref{Bvaltypesreal}.
\end{proof}

\section{A lattice invariant of models}
\label{9}

In this section we will study the left Kan extension of $Sub_{\mathcal{C}}:\mathcal{C}^{op}\to \mathbf{DLat}$ along $Y:\mathcal{C}^{op}\hookrightarrow \mathbf{Lex}(\mathcal{C},\mathbf{Set})$, resulting $L^1:\mathbf{Lex}(\mathcal{C},\mathbf{Set}) \to \mathbf{DLat}$ (where $\mathcal{C}$ is a coherent category). The simple definition implies nice functorial properties. Then, in the next section we will translate between model-theoretic properties of a model $M:\mathcal{C}\to Sh(B)$ and algebraic properties of the lattice $L^1(\Gamma M)$. 

\begin{definition}
    Take a coherent category $\mathcal{C}$, and an object $x\in \mathcal{C}$. We write $L_{\mathcal{C}}^x=L^x$ for the left Kan extension
\[
\adjustbox{scale=0.85}{
\begin{tikzcd}
	{\mathcal{C}^{op}} & {(\faktor{\mathcal{C}}{x})^{op}} &&& {\mathbf{DLat}} \\
	\\
	{\mathbf{Lex}(\mathcal{C},\mathbf{Set})}
	\arrow["{({!}^{-1})^{op}}", from=1-1, to=1-2]
	\arrow["Y"', hook', from=1-1, to=3-1]
	\arrow["{Sub_{\faktor{\mathcal{C}}{x}}}", from=1-2, to=1-5]
	\arrow["{L^x}"', from=3-1, to=1-5]
\end{tikzcd}
}
\]
\end{definition}

\begin{remark}
    The composite $Sub_{\faktor{\mathcal{C}}{x}}\circ ({!}^{-1})^{op}$ is $Sub_{\mathcal{C}}(x\times -):\mathcal{C}^{op}\to \mathbf{DLat}$.
\end{remark}

\begin{remark}
    Take a lex functor $F:\mathcal{C}\to \mathbf{Set}$. Recall from \cite[Section 3]{bvalued} that the category $\mathcal{C}_F=\colim _{(y,b)\in (\int F)^{op}}\faktor{\mathcal{C}}{y}$ is the extension of $\mathcal{C}$ by adding global elements (constant symbols) naming the elements of $F$. The explicit computation of the Kan extension above yields $L^x(F)=\colim _{(y,b)\in (\int F)^{op}}Sub_{\mathcal{C}}(x\times y)$. It follows that $L^x(F)=Sub_{\mathcal{C}_F}(x)$ and therefore we can think of $L^x(F)$ as the Lindenbaum-Tarski algebra of formulas below $x$ with parameters from $F$.
\end{remark}

\begin{remark}
    If $\kappa $ is weakly compact and $\mathcal{C}$ is $\kappa$-coherent, then $\restr{L^x}{\mathbf{Lex}_{\kappa }(\mathcal{C},\mathbf{Set})}$ lands in the category of $\kappa$-coherent distributive lattices.
\end{remark}

\begin{remark}
\label{lantoset}
    Since the forgetful functor $U:\mathbf{DLat}\to \mathbf{Set}$ preserves filtered colimits, we have $U\circ L^x=Lan_Y(U\circ Sub_{\mathcal{C}}(x\times -))$. A left Kan extension along $Y:\mathcal{C}^{op}\to \mathbf{Lex}(\mathcal{C},\mathbf{Set})$ is the same as the domain restriction of the left Kan extension along $Y:\mathcal{C}^{op}\to \mathbf{Set}^{\mathcal{C}}$. Putting these together yields $U\circ L^x=\restr{Lan_{Y}(U\circ Sub_{\mathcal{C}}(x\times -))}{\mathbf{Lex}(\mathcal{C},\mathbf{Set})}$.
\end{remark}

This has some implications on the functorial properties of $L^x$. 

\begin{proposition}
\label{presfiltcolim}
    Let $\mathcal{C}$ be coherent, and $x\in \mathcal{C}$. Then $L^x:\mathbf{Lex}(\mathcal{C},\mathbf{Set})\to \mathbf{DLat}$ preserves filtered colimits.
\end{proposition}

\begin{proof}
    $Lan_Y(U\circ Sub_{\mathcal{C}}(x\times - )):\mathbf{Set}^{\mathcal{C}}\to \mathbf{Set}$ preserves all colimits (it is a left adjoint), and therefore its domain restriction to $\mathbf{Lex}(\mathcal{C},\mathbf{Set})$ preserves filtered colimits. Finally, apply Remark \ref{lantoset}.
\end{proof}

\begin{theorem}
\label{coeqforfree}
    Let $\mathcal{C}$ be a $(\lambda ,\omega)$-pretopos for $\lambda > \omega $ (that is: a $\lambda $-geometric category with $\lambda $-small disjoint coproducts and with quotients of equivalence relations). Then $\mathcal{C}$ has $\lambda $-small colimits which are preserved by any $(\lambda ,\omega )$-coherent functor. Moreover, the functor $Sub_{\mathcal{C}}:\mathcal{C}^{op}\to \mathbf{DLat}$ preserves $\lambda $-small limits.
\end{theorem}

\begin{proof}
    The first part of the statement follows from Lemma 1.4.19.~of \cite{elephant}. Since $Sub_{\mathcal{C}}$ maps disjoint coproducts to products, we are left to show that it takes coequalizers to equalizers. This is equivalent to the following exactness property: if $f$ and $g$ in
    
\[\begin{tikzcd}
	x && y && z \\
	& pb && {pb?} \\
	{f^{-1}u=g^{-1}u} && u && \bullet
	\arrow["f", shift left=2, from=1-1, to=1-3]
	\arrow["g"', shift right=2, from=1-1, to=1-3]
	\arrow[hook, from=3-3, to=1-3]
	\arrow["{f'}", shift left=2, from=3-1, to=3-3]
	\arrow["{g'}"', shift right=2, from=3-1, to=3-3]
	\arrow[hook, from=3-1, to=1-1]
	\arrow[two heads, from=1-3, to=1-5]
	\arrow[two heads, from=3-3, to=3-5]
	\arrow[hook, from=3-5, to=1-5]
\end{tikzcd}\]
pull back the same subobject from $u$, then the image factorization of $u\to y\to z$ yields a pullback (for the coequalizer $y\to z$). (This is because effective epis are stable under pullback and effective epi-mono factorizations are unique. So if $u$ can be pulled back from a subobject of $z$ along the coequalizer then the pullback square must be of this form. Since $Sub_{\mathcal{C}}$ maps effective epis to injections, under these conditions the coequalizer is mapped to a mono, which equalizes $f^{-1}$ and $g^{-1}$ and which is the largest such subobject.)

This is satisfied in $Sh(\mathcal{C},\tau _{\lambda -coh})$ as here $Sub_{Sh(\mathcal{C})}$ is representable via the subobject classifier, i.e.~it is $Sh(\mathcal{C})(-,\Omega )$. But the Yoneda-embedding $\mathcal{C}\to Sh(\mathcal{C},E_{\lambda })$ preserves and reflects finite limits and finite colimits (as it is a $(\lambda ,\omega )$-coherent functor), so the same exactness property holds in $\mathcal{C}$.
\end{proof}

\begin{theorem}
\label{preslambdalim}
    Assume that $\mathcal{C}$ is
    \begin{itemize}
        \item[a)] a coherent category with finite disjoint coproducts
        \item[ ] or
        \item[b)] a $(\lambda ,\omega )$-pretopos for $\lambda >\omega $.
    \end{itemize}
   Take $x\in \mathcal{C}$. Then $L^x:\mathbf{Lex}(\mathcal{C},\mathbf{Set})\to \mathbf{DLat}$ preserves
    \begin{itemize}
        \item[a)] finite products
        \item[ ] or
        \item[b)] $\lambda $-small limits.
    \end{itemize}
\end{theorem}

\begin{proof}
    The slice $\faktor{\mathcal{C}}{x}$ is a coherent category with finite disjoint coproducts, or a $(\lambda ,\omega )$-pretopos respectively. Both finite products and $\lambda $-small limits form a sound doctrine in the sense of \cite{classification}. Since $U\circ Sub_{\faktor{\mathcal{C}}{x}}\circ ({!}^{-1})^{op}$ preserves these limits, so does its left Kan extension $Lan_Y(U\circ Sub_{\faktor{\mathcal{C}}{x}}\circ ({!}^{-1})^{op}):\mathbf{Set}^{\mathcal{C}}\to \mathbf{Set}$. Remark \ref{lantoset} completes the proof.
\end{proof}

\begin{corollary}
    If $\mathcal{C}$ is a $(\lambda ,\omega )$-pretopos for $\lambda >\omega $ and $x\in \mathcal{C}$, then $L^x:\mathbf{Lex}(\mathcal{C},\mathbf{Set})\to \mathbf{DLat}$ preserves reduced products over $\lambda $-small index sets.
\end{corollary}

\begin{proof}
    If $I$ is a $\lambda $-small set, $\mathcal{F}:\mathcal{P}(I)\to 2$ is a filter, and $F_i:\mathcal{C}\to \mathbf{Set}$ are lex functors, then the reduced product $\faktor{\prod _iF_i}{\mathcal{F}}$ is given by $\colim _{\{J\in \mathcal{F}\}^{op}}\prod _{i\in J}F_i$. We can apply Proposition \ref{presfiltcolim} and Theorem \ref{preslambdalim}.
\end{proof}

\begin{proposition}
\label{l0presregmono}
    Let $\mathcal{C}$ be coherent, and $x\in \mathcal{C}$. Then $L^x:\mathbf{Lex}(\mathcal{C},\mathbf{Set})\to \mathbf{DLat}$ takes regular monomorphisms to injective maps.
\end{proposition}

\begin{proof}
    Every regular monomorphism in $\mathbf{Lex}(\mathcal{C},\mathbf{Set})$ is the filtered colimit of regular monos of the form $\widehat{z_i}\xrightarrow{-\circ p_i} \widehat{y_i}$ where $p_i$ is effective epi in $\mathcal{C}$ (see the proof of \cite[Theorem 12]{pure}). By Proposition \ref{presfiltcolim} it is enough to see that these are preserved. But it is clear that if $p$ is effective epi then $(id_x\times p)^{-1}:Sub_{\mathcal{C}}(x\times z)\to Sub_{\mathcal{C}}(x\times y)$ is injective.
\end{proof}

\begin{corollary}
    If $\mathcal{C}$ is a coherent category in which monomorphisms are regular (e.g.~a pretopos) and $x\in \mathcal{C}$, then $L^x:\mathbf{Lex}(\mathcal{C},\mathbf{Set})\to \mathbf{DLat}$ preserves monomorphisms.
\end{corollary}

\begin{proof}
    By Example \cite[Remark 16]{pure} in $\mathbf{Lex}(\mathcal{C},\mathbf{Set})$ monomorphisms are regular.
\end{proof}

In what follows we will focus on $L^1(F)$, the Lindenbaum-Tarski algebra of closed formulas with parameters from $F$. (This was studied in \cite[Section 6]{interpret}, some of those arguments are extended here, from $\mathbf{Set}$-valued models to global sections of $Sh(B)$-valued models.) We start by giving an explicit description.

\begin{proposition}
\label{siminL0}
    We write $[u\hookrightarrow x^a]$ for the element of $L^1(F)$ represented by $u\hookrightarrow x$, living in $Sub_{\mathcal{C}}(x)$, indexed with $(x,a)$. Then $u\hookrightarrow x^a \sim v\hookrightarrow y^b$ iff there exists $w \hookrightarrow x\times y$ such that $(a,b)\in Fw $ and $w \cap (u\times y) =w \cap (x\times v)$.
\end{proposition}

\begin{proof}
    As $(\int F)^{op}$ is $\kappa $-filtered, we get $u\hookrightarrow x^a \sim v\hookrightarrow y^b$ iff there is $w\hookrightarrow z$, $c\in Fz$ and arrows $f:z\to x$, $g:z\to y$ with $Ff(c)=a$ and $Fg(c)=b$, such that $w=f^{-1}u=g^{-1}v$. Taking the effective epi-mono factorization of $(f,g):z\to x\times y$, and using that if $h$ is effective epi then $h^{-1}$ is injective, completes the argument. 
\end{proof}

\begin{example}
\label{noultrapr}
Given a distributive lattice $K$, every coherent functor $K\to \mathbf{Set}$ factors as $K\xrightarrow{P} \mathbf{2}\to \mathbf{Set}$, hence yields a prime filter $p=P^{-1}(1)\in Spec(K)$. Then $L^1P=\faktor{K}{p}=\faktor{K}{\sim _{p}}$ where $a\sim _{p}b$ iff $\exists x\in p: x\cap a =x\cap b$.
\end{example}


\begin{remark}
We can repeat the construction in the language of model theory. Let $T\subseteq \Sigma _{\omega \omega }^g $ be a coherent theory and $M$ be a model of $T$. Then $L^1M$ is the following distributive lattice: its underlying set is $\faktor{\{\langle \varphi (x_1,\dots x_k), (a_i\in M_{X_i}) \rangle \}}{\sim }$ where $\varphi $ is a positive existential $\Sigma $-formula, $\vec{a}$ is a tuple of elements having the appropriate sorts, and $\langle \varphi (x_1,\dots x_k), (a_i\in M_{X_i}) \rangle \sim \langle \psi (y_1,\dots y_l), (b_i\in M_{Y_i}) \rangle $ iff there's a positive existential $\chi (x_1,\dots x_k,y_1,\dots y_l)$ such that $T\models \varphi \wedge \chi \Leftrightarrow \psi \wedge \chi $ and $M\models \chi (a_1,\dots a_k,b_1,\dots b_l)$. Here $\vec{x}$ and $\vec{y}$ are disjoint (so if $y_i=x_j$ then we replace it with a fresh variable $y_i'$). 
\begin{multline*}
    [\langle \varphi (x_1,\dots x_k), (a_i\in M_{X_i}) \rangle ]\wedge [\langle \psi (y_1,\dots y_l), (b_i\in M_{Y_i}) \rangle ]= \\ [\langle \varphi \wedge \psi (x_1,\dots x_k,y_1,\dots y_l), (a_1,\dots a_k,b_1,\dots b_l) \rangle ]
\end{multline*}
and similarly for $\vee $. ($\vec{x}$ and $\vec{y}$ are still disjoint, the top element is $[\top ]$, the bottom element is $[\bot ]$. $[\langle \varphi (\vec{x}),\vec{a} \rangle ]\sim [\top ]$ iff $M\models \varphi (\vec{a})$.)
\end{remark}

Now we extend the definition of $L^1$ by allowing the domain to vary.

\begin{definition}
Define $\mathbf{Coh}_e\downarrow ^{lex}\mathbf{Set}$ to be the strict 2-category whose objects are lex functors $\mathcal{C}\to \mathbf{Set}$ (where $\mathcal{C}$ is a non-fixed small coherent category), the 1-cells are pairs $(M:\mathcal{C}\to \mathcal{D}, \eta _M:F\Rightarrow GM)$ where $M$ is a coherent functor and $\eta _M$ is an arbitrary natural transformation, and the 2-cells are elementary natural transformations $\alpha :M\Rightarrow N$ making 

\[\begin{tikzcd}
	{\mathcal{C}} \\
	&&&& {\mathbf{Set}} \\
	\\
	{\mathcal{D}}
	\arrow[""{name=0, anchor=center, inner sep=0}, "F", from=1-1, to=2-5]
	\arrow[""{name=1, anchor=center, inner sep=0}, "M"{description}, curve={height=18pt}, from=1-1, to=4-1]
	\arrow[""{name=2, anchor=center, inner sep=0}, "N"{description}, curve={height=-18pt}, from=1-1, to=4-1]
	\arrow[""{name=3, anchor=center, inner sep=0}, "G"', from=4-1, to=2-5]
	\arrow["{\eta _N}", curve={height=-6pt}, shorten <=7pt, shorten >=7pt, Rightarrow, from=0, to=3]
	\arrow["{\eta _M}"', curve={height=6pt}, shorten <=7pt, shorten >=7pt, Rightarrow, from=0, to=3]
	\arrow["\alpha", shorten <=7pt, shorten >=7pt, Rightarrow, from=1, to=2]
\end{tikzcd}\]
commutative. We have an evident forgetful 2-functor $\mathbf{Coh}_e\downarrow ^{lex}\mathbf{Set}\to \mathbf{Coh}_e$, to the 2-category of small coherent categories, coherent functors, and elementary natural transformations.
\end{definition}

\begin{definition}
\label{ldef}
We define a 2-functor $L^1:\mathbf{Coh}_e\downarrow ^{lex}\mathbf{Set} \to \mathbf{DLat}$ as follows: Given a 2-cell as pictured above, we take $L^1F$ to be the colimit
\[
colim((\int F)^{op}\to \mathcal{C}^{op}\xrightarrow{Sub_{\mathcal{C}}} \mathbf{DLat})
\]
and $L^1\eta _M =L^1\eta _N$ to be the map induced by the cocone, whose leg at $(x,a)\in (\int F)^{op}$ is 

\[
\adjustbox{scale=1}{
\begin{tikzcd}
	& {Sub_{\mathcal{C}}(x)^a} \\
	{Sub_{\mathcal{D}}(Mx)^{\eta _{M,x}a}} && {Sub_{\mathcal{D}}(Nx)^{\eta _{N,x}a}} \\
	& {colim((\int G)^{op}\to \mathcal{D}^{op}\xrightarrow{Sub_{\mathcal{D}}}\mathbf{DLat})}
	\arrow["M"', from=1-2, to=2-1]
	\arrow["N", from=1-2, to=2-3]
	\arrow["{\alpha _x^{-1}}"', from=2-3, to=2-1]
	\arrow[from=2-1, to=3-2]
	\arrow[from=2-3, to=3-2]
\end{tikzcd}
}
\]
\end{definition}

Indeed, we got an extension of our previous definition:

\begin{remark}
    $L^1$ restricted to the fiber over $\mathcal{C} $ is $L^1_{\mathcal{C}}$. That is, given $\eta :F\Rightarrow G$ in $\mathbf{Lex}(\mathcal{C},\mathbf{Set})$ we have $L^1_{\mathcal{C}}(\eta )=L^1((1_{\mathcal{C}},\eta ))$ (as their restrictions to $Sub_{\mathcal{C}}(x)^a$ coincide, namely it is $Sub_{\mathcal{C}}(x)^a \xrightarrow{id} Sub_{\mathcal{C}}(x)^{\eta _x a}\to L^1_{\mathcal{C}}G$).
\end{remark}

\begin{proposition}
\label{injsurj}
Let 
\[\begin{tikzcd}
	{\mathcal{C}} \\
	&& {\mathbf{Set}} \\
	{\mathcal{D}}
	\arrow[""{name=0, anchor=center, inner sep=0}, "F", from=1-1, to=2-3]
	\arrow["M"', from=1-1, to=3-1]
	\arrow[""{name=1, anchor=center, inner sep=0}, "G"', from=3-1, to=2-3]
	\arrow["\eta", shorten <=4pt, shorten >=4pt, Rightarrow, from=0, to=1]
\end{tikzcd}\]
be a 1-cell in $\mathbf{Coh}_e\downarrow ^{lex}\mathbf{Set} $. 

If $M$ is conservative, full wrt.~subobjects, and $\eta $ is elementary then $L^1((M,\eta ))$ is injective.

If $\mathcal{D}$ is $(\lambda , \omega )$-coherent, $M$ is full wrt.~subobjects and $<\lambda $-covers its codomain (i.e.~for $z\in \mathcal{D}$ there's a $<\lambda $ effective epimorphic family $(h_i:M x_i\to z)_{i}$), $G$ is $(\lambda ,\omega )$-coherent and $\eta $ is pointwise surjective then $L^1((M,\eta ))$ is surjective.
\end{proposition}

\begin{proof}
Injectivity: $[(u\xhookrightarrow{i} x)^a]$ is mapped to $[(Mu\xhookrightarrow{Mi} Mx)^{\eta _x a}]$. Assume $(Mu\xhookrightarrow{Mi} Mx)^{\eta _x a}\sim (Mv\xhookrightarrow{Mj} My)^{\eta _y b}$ that is, there's $\chi \hookrightarrow Mx\times My$ with $\chi \cap (Mx\times Mv)= \chi \cap (Mu\times My)$ and with $(\eta _x(a),\eta _y(b))\in G\chi $. As $M$ is full wrt.~subobjects, $\chi =Mw$, as $M$ is conservative $w\cap (x\times v)= w\cap (u\times y)$ and as $\eta $ is elementary $(a,b)\in Fw$.

Surjectivity: As $M$ is full wrt.~subobjects and $\eta $ is pointwise surjective it is enough to prove that any $(r \hookrightarrow z)^d$ is equivalent to some $( r'\hookrightarrow Mx)^{d'}$. The effective epimorphic family $(h_i:Mx_i\to z)_i$ is mapped to a covering by $G$, hence for some $i$ there's $d'\in GMx_i$ with $Gh_i(d')=d$. Then $(h_i^{-1}r \hookrightarrow Mx_i)^{d'}$ is equivalent to our original element. 
\end{proof}

\begin{remark}
    By \cite[Theorem 15]{pure} this gives a second proof for Proposition \ref{l0presregmono} when $x=1$. 
\end{remark}

\section{Positively closed models}

In positive model theory, one of the central concepts is that of a positively closed model. We now extend and study this notion in the setting of topos-valued positive model theory.

In \cite[Section 4]{bvalued} we proved that a $Sh(B)$-valued model $M:\mathcal{C}\to Sh(B)$ can be reconstructed from its global sections $\Gamma M:\mathcal{C}\to \mathbf{Set}$. So one might hope that model-theoretic properties of $M$ are reflected in the algebraic properties of our lattice $L^1(\Gamma M)$ from the previous section. 

In Theorems \ref{posclequiv}, \ref{strongposclequiv} and \ref{Bcompletestrongposcl} we will give characterizations of positively closed and strongly positively closed $Sh(B)$-valued models in terms of $L^1(\Gamma M)$, relying on the results from Section 3. 

In Examples \ref{Open(Q)} and \ref{Closed(X)} we will show that positively closed, but not strongly positively closed $Sh(B)$-valued models (of coherent theories) exist, unlike in the $\mathbf{Set}=Sh(2)$-valued case.

\begin{definition}
    Let $\kappa $ be weakly compact, $\mathcal{C}$ a $\kappa $-coherent category, $\mathcal{E}$ a $\kappa $-topos, and $M:\mathcal{C}\to \mathcal{E}$ a $\kappa $-coherent functor. We say that $M$ is \emph{positively closed} if given any $N:\mathcal{C}\to \mathcal{E}$ $\kappa $-coherent functor, every $M\Rightarrow N$ natural transformation is elementary. 
\end{definition}

\begin{remark}
    One can also define positively closed models more generally, with $\kappa $-sites in place of $\kappa $-coherent categories, see \cite{thesis}. See also \cite[Definition 4.1]{exclosedkamsma}.
\end{remark}

We will compare the above global notion to the following local one:

\begin{definition}
\label{strongposcldef}
   Let $\kappa $ be weakly compact, $\mathcal{C}$ a $\kappa $-coherent category, $\mathcal{E}$ a $\kappa $-topos, and $M:\mathcal{C}\to \mathcal{E}$ a $\kappa $-coherent functor. We say that $M$ is \emph{strongly positively closed} if for any $u\hookrightarrow x$ mono in $\mathcal{C}$ we have $Mx=Mu \cup \bigcup _{v\hookrightarrow x : v\cap u=\emptyset } Mv$.
\end{definition}

\begin{remark}
    In \cite[Definition 5.1]{exclosedkamsma} the authors define strongly positively closed inverse image functors $F^*:\mathcal{E}_1\to \mathcal{E}_2$ wrt.~a fixed generating set $\chi $ of $\mathcal{E}_1$ and a fixed collection of subobjects which generate the subobjects of objects in $\chi $ under taking unions. As in our case $\mathcal{E}=Sh(\mathcal{C},\tau _{\kappa -coh})$ for some $\kappa $-coherent $\mathcal{C}$, the generating set of objects and the generating set of subobjects are the ones from $\mathcal{C}$. 
\end{remark}

\begin{remark}
    Assume that $\kappa $ is weakly compact, $B$ is a $\kappa$-coherent Boolean algebra, $\mathcal{C}$ and $M:\mathcal{C}\to Sh(B,\tau _{\kappa -coh})$ are $\kappa$-coherent. 
    
    Then $M$ is strongly positively closed iff for any $u\hookrightarrow x$ mono in $\mathcal{C}$, $b\in B$ and $s\in Mx(b)$ we can write $b=b_1\vee b_2$ s.t.~$\restr{s}{b_1}\in Mu(b_1)$ and $\restr{s}{b_2}\in Mv(b_2)$ where $v\hookrightarrow x$ is such that $u\cap v=\emptyset $.
\end{remark}

\begin{remark}
    When $B=2$ strongly positively closed becomes the following separation property for definable sets:

    Given $\mathcal{C}$, $M:\mathcal{C}\to \mathbf{Set}$ $\kappa$-coherent, $M$ is strongly positively closed iff for any $u\hookrightarrow x$ and $a\in Mx\setminus Mu$ there is $v\hookrightarrow x$ s.t.~$u\cap v=\emptyset $ and $a\in Mv$.

    When $\mathcal{C}$ is the syntactic category of $T\subseteq L_{\omega \omega }^g$, this gives back the classical definition of a strongly positively closed model; whenever $M\not \models \varphi (\vec{a})$ for some positive existential formula $\varphi (\vec{x})$, we can find a positive existential formula $\psi (\vec{x})$, such that $M\models \psi (\vec{a})$ and $T\vdash \varphi (\vec{x})\wedge \psi (\vec{x})\to \bot $.
\end{remark}

\begin{proposition}
     Assume that $\kappa $ is weakly compact, $\mathcal{E}$ is a $\kappa $-topos, $\mathcal{C}$ and $M:\mathcal{C}\to \mathcal{E}$ are $\kappa $-coherent. If $M$ is strongly positively closed then it is positively closed. (See also \cite[Proposition 5.6]{exclosedkamsma})
\end{proposition}

\begin{proof}
    As in
\[\begin{tikzcd}
	Mu && Mx && {\bigcup_{v:\ v\cap u=\emptyset } Mv} \\
	Nu && Nx && {\bigcup_{v:\ v\cap u=\emptyset } Nv}
	\arrow[hook, from=1-1, to=1-3]
	\arrow["{\alpha _u}"', from=1-1, to=2-1]
	\arrow["{\alpha _x}", from=1-3, to=2-3]
	\arrow[hook', from=1-5, to=1-3]
	\arrow[from=1-5, to=2-5]
	\arrow[hook, from=2-1, to=2-3]
	\arrow[hook', from=2-5, to=2-3]
\end{tikzcd}\]
    both rows are disjoint unions, it follows that both squares are pullbacks.
\end{proof}

\begin{proposition}
\label{psi2}
    Let $\kappa $ be weakly compact, $\mathcal{C}$ a $\kappa$-coherent category with $\kappa $-small disjoint coproducts, $B$ a $\kappa$-coherent Boolean algebra, and $M:\mathcal{C}\to Sh(B,\tau _{\kappa -coh})$ a $\kappa$-coherent functor. Then there is an injective $\kappa $-coherent homomorphism $\psi  :B\to L^1(\Gamma M)$ given by $b\mapsto [1\sqcup \emptyset \hookrightarrow 1\sqcup 1^{(\restr{*}{b},\restr{*}{\neg b})}]$ which is an isomorphism onto the Boolean algebra of complemented elements $ L^1(\Gamma M)^{\neg }$.
\end{proposition}

\begin{proof}
    This is \cite[Proposition 4.3]{bvalued}.
\end{proof}

In particular, if $M:\mathcal{C}\to \mathbf{Set}$ is $\kappa$-coherent, then $L^1(M)$ has no non-trivial complemented elements. 

In this section our goal is to characterize model-theoretic properties of models $M:\mathcal{C}\to Sh(B)$, in terms of this homomorphism $\psi $.

\begin{theorem}
\label{LGMbalg}
    Let $\kappa $ be weakly compact, $B$ a $\kappa$-coherent Boolean algebra, $\mathcal{C}$ a $\kappa$-coherent category with $\kappa $-small disjoint coproducts, and $M:\mathcal{C}\to Sh(B,\tau _{\kappa -coh})$ a $\kappa$-coherent functor. Then $M$ is strongly positively closed iff $L^1(\Gamma M)$ is a Boolean algebra.
\end{theorem}

\begin{proof}
    \fbox{$\Rightarrow $} We want $[u\hookrightarrow x^s]$ to be complemented. By the assumption we have $\top=b_1\vee b_2$ with $\restr{s}{b_1}\in Mu(b_1)$ and $\restr{s}{b_2}\in Mv(b_2)$ for some $v\hookrightarrow x$ with $u\cap v=\emptyset $. Therefore $(u\sqcup v, s) \hookrightarrow (x,s)$ is a map in $\int \Gamma M$, consequently $[u\hookrightarrow x^s]=[u\hookrightarrow u\sqcup v^s]$ which is complemented.

    \fbox{$\Leftarrow $} Fix $u\hookrightarrow x$, $b\in B$, $s\in Mx(b)$. Write $\widetilde{s}=(\restr{s}{b},\restr{*}{\neg b})$ for the global section in $M(x\sqcup 1)(\top )$ glued from $s\in Mx(b)\hookrightarrow M(x\sqcup 1)(b)$ and $*\in M1(\neg b)\hookrightarrow M(x\sqcup 1)(\neg b)$. By our assumption $[u\sqcup \emptyset \hookrightarrow x\sqcup 1^{(\restr{s}{b},\restr{*}{\neg b})}]$ is complemented. That is, for some $h:r\to x\sqcup 1$ which hits $\widetilde{s}$, the pullback $h^{-1}(u\sqcup \emptyset )$ is complemented. By writing $r_1$ for $h^{-1}(x\sqcup \emptyset )$ and $r_2$ for $h^{-1}(\emptyset \sqcup 1)$ we have a pullback square
\[\begin{tikzcd}
	{u\sqcup \emptyset } && {x\sqcup 1 } \\
	\\
	{u^*\sqcup \emptyset } && {r_1\sqcup r_2 }
	\arrow[hook, from=1-1, to=1-3]
	\arrow[from=3-1, to=1-1]
	\arrow[hook, from=3-1, to=3-3]
	\arrow["{h_1\sqcup h_2}"', from=3-3, to=1-3]
\end{tikzcd}\]
such that $\widetilde{s}$ has a preimage $t\in M(r_1\sqcup r_2)(\top )$ and $u^*\sqcup \emptyset $ is complemented.

It follows easily that $\restr{t}{b}\in Mr_1(b)$ and $\restr{t}{\neg b}\in Mr_2(b)$. It is also clear that $u^*\hookrightarrow r_1$ is complemented, say $(u^*)^c=w$. Write $v$ for its image under $h_1$:
\[\begin{tikzcd}
	u && x && v \\
	{u^*} && {r_1} && w
	\arrow[""{name=0, anchor=center, inner sep=0}, hook, from=1-1, to=1-3]
	\arrow[hook', from=1-5, to=1-3]
	\arrow[from=2-1, to=1-1]
	\arrow[""{name=1, anchor=center, inner sep=0}, hook, from=2-1, to=2-3]
	\arrow["{h_1}"', from=2-3, to=1-3]
	\arrow[two heads, from=2-5, to=1-5]
	\arrow[hook', from=2-5, to=2-3]
	\arrow["pb"{description}, draw=none, from=0, to=1]
\end{tikzcd}\]
Since $Mr_1=Mu^* \sqcup Mw$ there is $b=b_1\sqcup b_2$ with $\restr{t}{b_1}\in Mu^*(b_1)$ and $\restr{t}{b_2}\in Mw(b_2)$, hence $\restr{s}{b_1}\in Mu(b_1)$ and $\restr{s}{b_2}\in Mv(b_2)$. It suffices to prove $u\cap v=\emptyset $, but that follows as $\exists _{h_1}(h_1^{-1}u\cap w)=u\cap \exists _{h_1}(w)$ (where $\exists _{h_1}:Sub_{\mathcal{C}}(r_1)\to Sub_{\mathcal{C}}(x)$ is taking the image along $h_1$).
    
\end{proof}

We have arrived to the following characterization of strongly positively closed models:

\begin{theorem}
\label{strongposclequiv}
    Let $\kappa $ be weakly compact, $B$ a $\kappa$-coherent Boolean algebra, $\mathcal{C}$ a $\kappa$-coherent category with $\kappa $-small disjoint coproducts, and $M:\mathcal{C}\to Sh(B,\tau _{\kappa -coh})$ a $\kappa$-coherent functor. Then the following are equivalent:
\begin{enumerate}
    \item $M$ is strongly positively closed.
    \item $L^1(\Gamma M)$ is a Boolean algebra.
    \item $L^1(\Gamma M)\cong B$.
\end{enumerate}
    
\end{theorem}

\begin{proof}
    By Proposition \ref{psi2} and Theorem \ref{LGMbalg}.
\end{proof}

Our next goal is to find a similar characterization of positive closedness in terms of the lattice $L^1(\Gamma M)$. Note that since $L^1(\Gamma M)$ is computed from the model $M$ itself, such a characterization implies that the global property of being positively closed is equivalent to a local condition.

First we make a general observation:

\begin{theorem}
    Let $\kappa $ be weakly compact, $\mathcal{C}$ a $\kappa$-coherent category with $\kappa $-small disjoint coproducts, and $F:\mathcal{C}\to \mathbf{Set}$ be $\kappa $-regular. Let $B$ denote the Boolean algebra of complemented subobjects in $L^1F$. Then for any $b\in B$ and $x_i\in L^1F$ ($i<\gamma <\kappa $), $b\leq \bigcup _i x_i$ implies $b=\bigcup _i b_i$ such that $b_i\in B$, $b_i\leq x_i$. In other terms $\downarrow :L^1F\to (Id_{\kappa }(B),\subseteq )$ defined by $x\mapsto \{b: b\leq x\}$ is a $\kappa$-homomorphism.
\end{theorem}

\begin{proof}
    We have 
\[\begin{tikzcd}
	B && {Sh(B,\tau _{\kappa -coh})} && {Sh(B,\tau _{\kappa -coh})} \\
	& {=} \\
	{\mathcal{C}_F} && {Sh(\mathcal{C}_F,\tau _{\kappa -coh})} && {Sh(\mathcal{C}_F,\tau _{\kappa -coh})}
	\arrow["{\iota }"', from=1-1, to=3-1]
	\arrow["Y", hook, from=1-1, to=1-3]
	\arrow["Y"', hook, from=3-1, to=3-3]
	\arrow["{\iota ^*}"', from=1-3, to=3-3]
	\arrow[""{name=0, anchor=center, inner sep=0}, Rightarrow, no head, from=1-3, to=1-5]
	\arrow[""{name=1, anchor=center, inner sep=0}, Rightarrow, no head, from=3-3, to=3-5]
	\arrow["{\iota _*}"', from=3-5, to=1-5]
	\arrow["\eta", shorten <=9pt, shorten >=9pt, Rightarrow, from=0, to=1]
\end{tikzcd}\]
 By \cite[Corollary 5.13]{bvalued} $\eta $ is elementary. Therefore since $\eta _{\widehat{\top }}$ is iso each $\eta _{\widehat{b}}$ is iso, so  $\restr{\eta }{B}$ is an isomorphism. 

As all three maps $Y$, $\iota $ and $\iota _*Y$ are $\kappa$-coherent (the last one by \cite[Theorem 4.1.ii)]{bvalued}), we get a factorization
\[\begin{tikzcd}
	B && {L^1F} && {Id_{\kappa }(B)=Sub_{Sh(B)}(1)}
	\arrow["{\iota }"', from=1-1, to=1-3]
	\arrow["r"', dashed, from=1-3, to=1-5]
	\arrow["Y", curve={height=-18pt}, from=1-1, to=1-5]
\end{tikzcd}\]

Recall that a subsheaf of $B(-,\top )$ is the same as a $\kappa $-closed ideal (downward closed subset, closed under $\kappa $-small unions, contains $\bot $), and $Y$ is the same as the canonical map $\downarrow :B\to (Id_{\kappa }(B),\subseteq )$. Given a subobject $x$ of $1\in \mathcal{C}_F$, the composite $\iota _*Y$ takes it to $\restr{\mathcal{C}_F(-,x)}{B^{op}}$, hence also $r=\downarrow $.
\end{proof}

\begin{remark}
    An alternative proof is to observe that for $\beta =tp $ we get $\mathcal{C}_{M,\beta }=\mathcal{C}_{\Gamma M}$ and then use Proposition \ref{extendsfromB}.
\end{remark}

We proved that if $\kappa $ is weakly compact, $B$ is a $\kappa$-coherent Boolean algebra, $\mathcal{C}$ is a $\kappa$-coherent category with $\kappa $-small disjoint coproducts, $M:\mathcal{C}\to Sh(B)$ is $\kappa $-coherent then $\downarrow :B\to Id_{\kappa }(B)$ factors as  $B\hookrightarrow L^1(\Gamma M)\xrightarrow{\downarrow }Id_{\kappa }(B)$. Now we shall prove that such ``retractions'' correspond to natural transformations $\alpha :M\Rightarrow S_{\mathcal{C}}^{Sh(B)}$ and that $\downarrow $ corresponds to $tp$. (For much simpler proofs in the $B=2$ and $\kappa =\omega $ case see \cite[Remark 6.7 and Proposition 6.17]{interpret}.)

\begin{theorem}
\label{L0classifnattr}
    Let $\kappa $ be weakly compact, $\mathcal{C}$ a $\kappa$-coherent category with $\kappa $-small disjoint coproducts, $B$ a $\kappa$-coherent Boolean algebra, and $M:\mathcal{C}\to Sh(B,\tau _{\kappa -coh})$ a $\kappa$-coherent functor. Then there is a bijection between $r:L^1(\Gamma M)\to Id_{\kappa }(B)$ $\kappa $-homomorphisms satisfying $\restr{r}{B}=\downarrow $ and $\alpha : M\Rightarrow S_{\mathcal{C}}^{Sh(B)}$ natural transformations. 
    
    Given $r$, the corresponding map has components 
    \[\alpha ^r_{x,b}:Mx(b)\to \mathbf{DLat}_{\kappa ,\kappa }(Sub_{\mathcal{C}}(x),Id_{\kappa }(\downarrow b))
    \]
    taking $s\in Mx(b)$ to $u\hookrightarrow x \ \mapsto \ r([u\sqcup \emptyset \hookrightarrow x\sqcup 1^{(\restr{s}{b},\restr{*}{\neg b})}])$. 
    
    Conversely, given $\alpha $, the corresponding homomorphism $r^{\alpha }$ takes $[u\hookrightarrow x^s]$ to $\alpha _{x,\top }(s)(u)$.

\end{theorem}

\begin{proof}
First we prove that $r\mapsto \alpha ^r$ is well-defined:

 $\alpha ^r_{x,b}(s)(u)=r([u\sqcup \emptyset \hookrightarrow x\sqcup 1^{(\restr{s}{b},\restr{*}{\neg b})}])\subseteq \ \downarrow b$: this follows as
 \[
[u\sqcup \emptyset \hookrightarrow x\sqcup 1^{(\restr{s}{b},\restr{*}{\neg b})}]\leq [x\sqcup \emptyset \hookrightarrow x\sqcup 1^{(\restr{s}{b},\restr{*}{\neg b})}]=[1\sqcup \emptyset \hookrightarrow 1\sqcup 1^{(\restr{*}{b},\restr{*}{\neg b})}]
 \]
 and the right side is taken to $\downarrow b$, and $r$ is $\leq $-preserving.

 $\alpha ^r_{x,b}$ is a $\kappa $-homomorphism: this is immediate from $r$ being a $\kappa $-homomorphism.

    $\alpha ^r_x$ is a natural transformation: given $b'\leq b$ it follows that 
    \begin{multline*}
    [u\sqcup \emptyset \hookrightarrow x\sqcup 1 ^{(\restr{s}{b},\restr{*}{\neg b})}]\wedge [(1\sqcup \emptyset \hookrightarrow 1\sqcup 1 ^{(\restr{*}{b'},\restr{*}{\neg {b'}})}]= \\ =[u\sqcup \emptyset \sqcup \emptyset \sqcup \emptyset \hookrightarrow x\sqcup 1 \sqcup x\sqcup 1 ^{(\restr{s}{b'},-,\restr{s}{b\setminus b'},\restr{*}{\neg {b'}}) }]= [u\sqcup \emptyset \hookrightarrow x\sqcup 1 ^{(\restr{s}{b'},\restr{*}{\neg {b'}})}]
    \end{multline*}
    and hence
    \begin{multline*}
        r([u\sqcup \emptyset \hookrightarrow x\sqcup 1 ^{(\restr{s}{b},\restr{*}{\neg b})}])\cap \downarrow b' = \\ = r([u\sqcup \emptyset \hookrightarrow x\sqcup 1 ^{(\restr{s}{b},\restr{*}{\neg b})}]\wedge [(1\sqcup \emptyset \hookrightarrow 1\sqcup 1 ^{(\restr{*}{b'},\restr{*}{\neg {b'}})}])= r([u\sqcup \emptyset \hookrightarrow x\sqcup 1 ^{(\restr{s}{b'},\restr{*}{\neg {b'}})}])
    \end{multline*}
    which is precisely what we want.

    $\alpha ^r$ is a natural transformation: the commutativity of 
\[\begin{tikzcd}
	{Mx(b)} && {My(b)} \\
	\\
	{\mathbf{DLat}_{\kappa ,\kappa }(Sub_{\mathcal{C}}(x),Id_{\kappa }(\downarrow b))} && {\mathbf{DLat}_{\kappa ,\kappa }(Sub_{\mathcal{C}}(y),Id_{\kappa }(\downarrow b))}
	\arrow["{Mf_b}", from=1-1, to=1-3]
	\arrow["{\alpha _{x,b}^r}"', from=1-1, to=3-1]
	\arrow["{\alpha _{y,b}^r}", from=1-3, to=3-3]
	\arrow["{-\circ f^{-1}}", from=3-1, to=3-3]
\end{tikzcd}\]
    means that given $s\in Mx(b)$ and $v\hookrightarrow y$, we have $r([f^{-1}v\sqcup \emptyset \hookrightarrow x\sqcup 1^{(\restr{s}{b},\restr{*}{\neg b})}]) = r([v\sqcup \emptyset \hookrightarrow y\sqcup 1^{(\restr{Mf_b(s)}{b},\restr{*}{\neg b})}])$. This follows as the two elements of $L^1(\Gamma M)$ are equal, witnessed by $f\sqcup 1:x\sqcup 1\to y\sqcup 1$ (since $M(f\sqcup 1)_{\top }$ takes $(\restr{s}{b},\restr{*}{\neg b})$ to $(\restr{Mf_b(s)}{b},\restr{*}{\neg b})$ and the pullback of $v\sqcup \emptyset $ along $f\sqcup 1$ is $f^{-1}v\sqcup \emptyset $).

Now we prove that $\alpha \mapsto r^{\alpha }$ is well-defined:

    $r^{\alpha }$ is well-defined on equivalence classes: it suffices to check that $r^{\alpha }([f^{-1}v\hookrightarrow x^s])=r^{\alpha }([v\hookrightarrow y^{Mf_{\top }(s)}])$ (as these pairs generate the equivalence relation). That is, we need $\alpha _{y,\top }(Mf_{\top }(s))(v)=\alpha _{x,\top }(s)(f^{-1}(v))$ which is the naturality of $\alpha $ in $f$.

    $r^{\alpha }$ is a homomorphism: given $<\kappa $ many objects in $L^1(\Gamma M)$ we can represent them as $[u_i\hookrightarrow x^s]$ (since $(\int \Gamma M)^{op}$ is $\kappa $-filtered), and their meet/join is preserved since $\alpha _{x,\top }(s)$ is a homomorphism.

    $\restr{r^{\alpha }}{B}=\downarrow $: given $[1\sqcup \emptyset \hookrightarrow 1\sqcup 1^{(\restr{*}{b},\restr{*}{\neg b})}]$ the diagram 
\[\begin{tikzcd}
	{Id_{\kappa }(\downarrow b)} && {Id_{\kappa }(B)} && {Id_{\kappa }(\downarrow \neg b)} \\
	\\
	{Sub_{\mathcal{C}}(1\sqcup \emptyset )} && {Sub_{\mathcal{C}}(1\sqcup 1)} && {Sub_{\mathcal{C}}(\emptyset \sqcup 1)}
	\arrow[from=1-3, to=1-1]
	\arrow[from=1-3, to=1-5]
	\arrow["{\alpha _{1,b}(\restr{*}{b})}", from=3-1, to=1-1]
	\arrow["{\alpha _{1\sqcup 1,b}(\restr{*}{b},-)}"{description, pos=0.3}, color={rgb,255:red,128;green,128;blue,128}, from=3-3, to=1-1]
	\arrow["{\alpha _{1\sqcup 1,\top}((\restr{*}{b},\restr{*}{\neg b}))}"{description, pos=0.6}, from=3-3, to=1-3]
	\arrow["{\alpha _{1\sqcup 1,\neg b}(-,\restr{*}{\neg b})}"{description, pos=0.3}, color={rgb,255:red,128;green,128;blue,128}, from=3-3, to=1-5]
	\arrow[from=3-3, to=3-1]
	\arrow[from=3-3, to=3-5]
	\arrow["{\alpha _{1,\neg b}(\restr{*}{\neg b})}"', from=3-5, to=1-5]
\end{tikzcd}\]
    commutes, as $\alpha _{x,b}$ is natural both in $x$ and in $b$. Therefore $\alpha _{1\sqcup 1,\top }((\restr{*}{b},\restr{*}{\neg b}))(1\sqcup \emptyset )$ equals $\downarrow b$ (since its intersection with $\downarrow b$ is $\downarrow b$, and its intersection with $\downarrow \neg b$ is $\emptyset $).

    Now we prove that these maps are inverses to each other:
    
    $\alpha \mapsto r\mapsto \alpha $ yields identity: the $(x,b)$-component of the resulting natural transformation takes a section $s\in Mx(b)$ to the function $u\hookrightarrow x \ \mapsto \  \alpha _{x\sqcup 1,\top }(\restr{s}{b},\restr{*}{\neg b})(u\sqcup \emptyset )=\alpha _{x,b}(s)(u)$. 

    $r\mapsto \alpha \mapsto r$ yields identity: the resulting map takes $[u\hookrightarrow x^s]$ to $r([u\sqcup \emptyset \hookrightarrow x\sqcup 1 ^{(\restr{s}{\top },\restr{*}{\bot })}])=r([u\hookrightarrow x^s])$.
\end{proof}

\begin{remark}
\label{tpisdownarrow}
    In the above bijection $\downarrow :L^{1}(\Gamma M)\to Id_{\kappa }(B)$ corresponds to $tp:M\Rightarrow S_{\mathcal{C}}^{Sh(B)}$. Indeed, we need that for $b'\leq b$:
    \[
    [1\sqcup \emptyset \hookrightarrow 1\sqcup 1 ^{(\restr{*}{b'}, \restr{*}{\neg b'})}]\leq [u\sqcup \emptyset \hookrightarrow x\sqcup 1^{(\restr{s}{b},\restr{*}{\neg b})}]
    \]
    iff $\restr{s}{b'}\in Mu(b')$. This is clear: the inequality can be written as 
    \[
    [x\sqcup \emptyset \sqcup \emptyset \hookrightarrow x\sqcup x\sqcup 1^{(\restr{s}{b'}, \restr{s}{b\setminus b'}, \restr{*}{\neg b})}]\leq  [u\sqcup u \sqcup \emptyset \hookrightarrow x\sqcup x\sqcup 1^{(\restr{s}{b'}, \restr{s}{b\setminus b'}, \restr{*}{\neg b})}]
    \]
    which holds iff we are allowed to take the pullback of these subobjects along $u\sqcup x\sqcup 1\hookrightarrow x\sqcup x\sqcup 1$, that is: iff $\restr{s}{b'}\in Mu(b')$.
\end{remark}

\begin{theorem}
\label{posclequiv}
     Let $\kappa $ be weakly compact, $B$ a $\kappa$-coherent Boolean algebra, $\mathcal{C}$ a $\kappa$-coherent category with $\kappa $-small disjoint coproducts, $M:\mathcal{C}\to Sh(B,\tau _{\kappa -coh})$ a $\kappa$-coherent functor. Assume that $Sh(B,\tau _{\kappa -coh})$ can realize $\kappa $-types (cf.~Definition \ref{realtypes}). Then the following are equivalent:
     \begin{enumerate}
         \item $M $ is positively closed.
         \item There is a unique natural transformation $M\Rightarrow S_{\mathcal{C}}^{Sh(B)}$ (namely: $tp_M$).
         \item There is a unique $\kappa $-homomorphism $L^1(\Gamma M)\to Id_{\kappa }(B)$ whose restriction to $B=L^1(\Gamma M)^{\neg }$ coincides with $\downarrow $ (namely: $\downarrow $).
     \end{enumerate}
\end{theorem}

\begin{proof}
    By Proposition \ref{elemifftpcommutes} $M$ is positively closed iff for each $N:\mathcal{C}\to Sh(B)$ $\kappa $-lex $E$-preserving and each $\alpha :M\Rightarrow N$ natural transformation $tp_M=tp_N\circ \alpha$, so $2\Rightarrow 1$ is clear. For the converse assume that $\beta :M\Rightarrow S_{\mathcal{C}}$ does not equal $tp_M$. Then by Proposition \ref{tpisminimal} $tp_M\subsetneqq \beta $, and by Theorem \ref{factorbeta} there is a $\kappa$-coherent $N:\mathcal{C}\to Sh(B)$ and $\alpha :M\Rightarrow N$ with $\beta \subseteq tp_N\circ \alpha $. So $tp_M\neq tp_N\circ \alpha $ and hence $M$ is not positively closed.

    $2\Leftrightarrow 3$ follows from Theorem \ref{L0classifnattr}.
\end{proof}

\begin{remark}
    We don't need that $Sh(B,\tau _{\kappa -coh})$ can realize $\kappa $-types to get $2\Leftrightarrow 3 \Rightarrow 1$.
\end{remark}

\begin{corollary}
    Let $\kappa $ be strongly compact (including $\kappa =\omega $), $\mathcal{C}$ a $\kappa$-coherent category, and $M:\mathcal{C}\to \mathbf{Set}$ a $\kappa$-coherent functor. Then $M$ is positively closed iff it is strongly positively closed.
\end{corollary}

\begin{proof}
    We can assume that $\mathcal{C}$ has $\kappa $-small disjoint coproducts.
    
    By Theorem \ref{setrealtypes} both Theorem \ref{strongposclequiv} and Theorem \ref{posclequiv} applies. If $L^1(\Gamma M)\neq 2$ then there are more than one $\kappa $-complete prime filters on it.
\end{proof}

Our next example shows that this equivalence fails in the $Sh(B)$-valued case, even for $\kappa =\omega $.

\begin{example}
\label{Open(Q)}
    Let $K$ be the distributive sub-lattice of $Open(\mathbb{Q})$ formed by finite unions of intervals/half-lines. Write $B$ for the Boolean algebra of complemented elements (that is: finite unions of intervals/half-lines with irrational endpoints). (From now on let's call half-lines intervals too.)

    $M=\ \downarrow :K\to Id(B)\hookrightarrow Sh(B,\tau _{\omega -coh})$ defined as $u\mapsto \{b\in B: b\subseteq u\}$ is a homomorphism. It suffices to prove that when $u$ and $v$ are both intervals then their union is preserved, and this is clear: if some irrational-endpoint $b$ is contained in $u\cup v$ then we can write $b$ as $b_1\cup b_2$ with $b_1\subseteq u$, $b_2\subseteq v$, $b_1,b_2$ complemented.

    $M$ is positively closed: Let $H:K\to Id(B)$ be a homomorphism and $\alpha : M \Rightarrow H$ be a natural transformation. In particular for each $u\in K$ we have $\alpha _u: M (u) \hookrightarrow H(u)$. When $b$ is complemented $M( b) \subseteq H(b)$ and $M(\neg b )\subseteq H(\neg b)$ implies that $\alpha _b$ is the identity. If for $u\in K$ $H(u)$ contains some irrational-endpoint interval which is not inside $u$ then it contains one which is disjoint from $u$, say $b$. We get $H(b)=\downarrow b \subseteq H(u)$ and $H(b)\cap H(u)=\emptyset $, contradiction. So $\alpha $ is identity.

    $M$ is not strongly positively closed: take $(-\infty ,p)\subseteq \mathbb{Q}$ with $p$ rational. If $M=\downarrow $ was strongly positively closed then we were be able to cover the irrational-endpoint interval $(p-\varepsilon ,p+\varepsilon )$ with two irrational-endpoint intervals, one lying inside $(-\infty ,p)$, the other being disjoint from it. This is not possible.
\end{example}

We close the section by showing that when $B$ is complete, $Sh(B,\tau _{\omega -coh})$-valued positively closed coherent functors admit an internal characterization. 

\begin{theorem}
\label{uniqueretract}
    Let $L$ be a distributive lattice and $B$ a sub-lattice which is a complete Boolean algebra. Take $x\in L$ and write $x^{-}=\bigcup \{b\in B : b\leq x\}$ and $x^{+}=\bigcap \{b\in B : b\geq x\}$. Then for any $r\in B $ with $x^{-}\leq r\leq x^{+}$ there is a retract map 
\[\begin{tikzcd}
	B & B \\
	L
	\arrow[Rightarrow, no head, from=1-1, to=1-2]
	\arrow[hook', from=1-1, to=2-1]
	\arrow["H"', from=2-1, to=1-2]
\end{tikzcd}\]
with $H(x)=r$.
\end{theorem}

\begin{proof}
    First note that the sub-lattice $\langle B\cup \{x\} \rangle \subseteq L$ generated by $B\cup \{x\}$ consists of elements of the form $(b_0 \wedge x )\vee b_1$. These are clearly there, and it is elementary to check that their set is closed under finite unions and finite intersections. 

    It suffices to prove that the function $\langle B\cup \{x\} \rangle \to B$ sending the element $(b_0 \wedge x )\vee b_1$ to $(b_0 \wedge r )\vee b_1$ is well-defined. Then it follows that it is a homomorphism, and since complete Boolean algebras are injective objects in the category of distributive lattices (see \cite[Theorem 3.2]{injdlat}), it extends further to $L$.

    So assuming $(b_0 \wedge x )\vee b_1 = (b_0' \wedge x )\vee b_1'$ we shall prove $(b_0 \wedge r )\vee b_1 = (b_0' \wedge r )\vee b_1'$. The latter is equivalent to the conjunction of
    \[
    ((b_0 \wedge r )\vee b_1 ) \wedge \neg \left((b_0' \wedge r )\vee b_1'\right)= (b_0\wedge r\wedge \neg b_0' \wedge \neg b_1') \vee (b_1\wedge \neg b_0' \wedge \neg b_1') \vee (b_1\wedge \neg r \wedge \neg b_1')=0
    \]
    and 
    \[
    \neg ((b_0 \wedge r )\vee b_1 ) \wedge \left((b_0' \wedge r )\vee b_1'\right)= (b_0'\wedge r\wedge \neg b_0 \wedge \neg b_1) \vee (b_1'\wedge \neg b_0 \wedge \neg b_1) \vee (b_1'\wedge \neg r \wedge \neg b_1)=0
    \]
    So what we have to prove is the set of inequalities
    \begin{enumerate}
        \item[] $r\leq \neg b_0 \vee b_0' \vee b_1'$
        \item[] $b_1\leq b_0' \vee b_1'$
        \item[] $b_1\wedge \neg b_1'\leq r$
        \item[] $r\leq \neg b_0' \vee b_0 \vee b_1$
        \item[] $b_1'\leq b_0 \vee b_1$
        \item[] $b_1'\wedge \neg b_1\leq r$
    \end{enumerate}
    All of these easily follow from the assumption if we write $x$ in place of $r$, but then just use $x^{-}\leq r \leq x^{+}$.
    
\end{proof}

\begin{theorem}
\label{Bcompletestrongposcl}
    Let $\mathcal{C}$ be a coherent category with finite disjoint coproducts, $B$ a complete Boolean algebra, and $M:\mathcal{C}\to Sh(B,\tau _{\omega -coh})$ a coherent functor. Then the following are equivalent:
\begin{enumerate}
    \item $M$ is positively closed.
    \item For every $u\hookrightarrow x$ in $\mathcal{C}$ and for every $s\in Mx(\top )$
    \begin{itemize}
        \item[-]there is a largest $b_0\in B$ with $\restr{s}{b_0}\in Mu(b_0)$, and
        \item[-]$\neg b_0=\bigcup \{b\in B : \exists v\hookrightarrow x : u\cap v=\emptyset \text{ and } \restr{s}{b}\in Mv(b) \}$.
    \end{itemize}
\end{enumerate}    
\end{theorem}

\begin{proof}
  $1\Rightarrow 2:$  By Theorem \ref{Bvaltypesreal} and by Theorem \ref{posclequiv} $M$ is positively closed iff there is a unique map
\[\begin{tikzcd}
	B & {Id(B)} \\
	{L^1(\Gamma M)}
	\arrow[hook, from=1-1, to=1-2]
	\arrow[hook', from=1-1, to=2-1]
	\arrow[dashed, from=2-1, to=1-2]
\end{tikzcd}\]
    making the triangle commute (namely: $\downarrow $). By Theorem \ref{uniqueretract} this implies that for every $\tau =[u\hookrightarrow x^s]\in L=L^1(\Gamma M)$ there's 
    \[\tau ^{-}=\bigcup \{b\in B : b\leq \tau \}=\tau ^{+}= \bigcap \{b \in B: b\geq \tau \}
    \]
    
    Moreover, since $B$ is an injective object in the category of distributive lattices, the mono $B\hookrightarrow L$ splits. If $M$ is positively closed then this implies that $\downarrow $ factors through $B$, and therefore for every $\tau \in L$ there is a $b_0\in B$ with $b\leq \tau $ iff $b\leq b_0$. Hence $\tau ^{-}=b_0$. (In particular $\tau ^{-}\leq \tau $. This is not obvious, as we form the union in $B$.)
    
    By Remark \ref{tpisdownarrow} $b\leq \tau $ iff $\restr{s}{b}\in Mu(b)$. So $b_0$ is the largest element in $B$ satisfying $\restr{s}{b}\in Mu(b)$. Moreover its complement is $\bigcup \{b: b\cap \tau =\emptyset \}$. 
    
    We are left to prove that $[u\hookrightarrow x^s]\cap [1 \sqcup \emptyset \hookrightarrow 1\sqcup 1 ^{(\restr{*}{b},\restr{*}{\neg b})}]=\emptyset $ iff there is a $v\hookrightarrow x$ with $v\cap u=\emptyset $ and $\restr{s}{b}\in Mv(b)$. This is a standard computation: represent the elements in question as $[u\sqcup u\hookrightarrow x\sqcup x ^{(\restr{s}{b},\restr{s}{\neg b})}]$ and  $[x\sqcup \emptyset \hookrightarrow x\sqcup x ^{(\restr{s}{b},\restr{s}{\neg b})}]$. Their intersection is $[u\sqcup \emptyset \hookrightarrow x\sqcup x ^{(\restr{s}{b},\restr{s}{\neg b})}]$ which is zero iff there is $w\hookrightarrow x\sqcup x$ such that $(\restr{s}{b},\restr{s}{\neg b})\in Mw(\top )$ and $w\cap (u\sqcup \emptyset )=\emptyset $. $w\hookrightarrow x\sqcup x$ is of the form $w_1\sqcup w_2\hookrightarrow x\sqcup x$ and we can take $v=w_1$.

    $2\Rightarrow 1:$ Since every retract map $H$ satisfies $\tau ^{-}\leq H(\tau )\leq \tau ^{+}$, if the equality holds then there's at most one $H$. But at least one exists ($\downarrow $). So $M$ is positively closed iff for every $\tau \in L$: $\tau ^{-}=\tau ^{+}$. But $2$ implies that $\tau ^{-}=b_0=\neg (\bigcup \{b: b\cap \tau =\emptyset \})=\tau ^{+}$.
\end{proof}

\begin{theorem}
\label{counterexdies}
    Let $\mathcal{C}$ be coherent with finite disjoint coproducts, $B$  a complete Boolean algebra, and $M:\mathcal{C}\to Sh(B,\tau _{\omega -coh})$ a coherent functor. Assume that $M$ is positively closed. Then $\widetilde{M}:\mathcal{C}\xrightarrow{M}Sh(B,\tau _{\omega -coh})\xrightarrow{id^*}Sh(B,\tau _{can})$ is strongly positively closed.
\end{theorem}

\begin{proof}
   First note that $id^*$ is just sheafification wrt.~$\tau _{can}$. Fix $s:\widehat{b}\to \widetilde{M}x=Mx^{\#}$. There is a cover $(b_i\to b)_i$ with each $s_i:\widehat{b_i}\to \widehat{b}\xrightarrow{s}Mx^{\#}$ factoring through $Mx\to Mx^{\#}$ (see \cite[Lemma 5.2]{presheaftype}). By Theorem \ref{Bcompletestrongposcl} for each $i$ there is a cover $(b_{i,j}\to b_i)_j$ such that $(\restr{s_i}{b_{i}},\restr{*}{\neg b_{i}})\in M(x\sqcup 1)$ restricted to $b_{i,0}$ lies in $M(u\sqcup \emptyset )(b_{i,0})\subseteq M(x\sqcup 1)(b_{i,0})$, and when restricted to $b_{i,j}$ with $j\neq 0$, it lies in some $M(v\sqcup 1)(b_{i,j})\subseteq M(x\sqcup 1)(b_{i,j})$ where $v\cap u=\emptyset $. That is, we have
\[
\adjustbox{scale=0.9}{
\begin{tikzcd}
	&& {\widehat{b}} && {Mx^{\#}} \\
	&&& {Mu^{\#}} && {Mv^{\#}} \\
	& {\widehat{b_i}} &&& Mx \\
	{\widehat{b_{i,0}}} &&& Mu && Mv & {Mv'} & \dots \\
	& {\widehat{b_{i,j}}}
	\arrow["s"{description}, from=1-3, to=1-5]
	\arrow[from=2-4, to=1-5]
	\arrow[from=2-6, to=1-5]
	\arrow[from=3-2, to=1-3]
	\arrow["{s_i}"{description}, dashed, from=3-2, to=3-5]
	\arrow[from=3-5, to=1-5]
	\arrow[from=4-1, to=3-2]
	\arrow[dashed, from=4-1, to=4-4]
	\arrow[from=4-4, to=2-4]
	\arrow[from=4-4, to=3-5]
	\arrow[from=4-6, to=2-6]
	\arrow[from=4-6, to=3-5]
	\arrow[from=4-7, to=3-5]
	\arrow[from=5-2, to=3-2]
	\arrow[curve={height=18pt}, dashed, from=5-2, to=4-6]
\end{tikzcd}
}
\]
showing that $Mx^{\#}=Mu^{\#}\cup \bigcup _{v:\ u\cap v=\emptyset} Mv^{\#}$.
   
\end{proof}

\begin{question}
\label{Bvalposcliffstrongly}
    Assume that $\mathcal{C}$ is a coherent category, $B$ is a complete Boolean algebra, and $M:\mathcal{C}\to Sh(B,\tau _{can})$ is a coherent functor. Is it true that $M$ is positively closed iff it is strongly positively closed?
\end{question}

Now we will give examples of $Sh(B,\tau _{\omega -coh})$-valued positively closed, not strongly positively closed models, where $B$ is complete. First we eliminate the necessity of disjoint coproducts.

\begin{proposition}
    Let $\mathcal{C}$ be a coherent category, and write $B=Sub_{\mathcal{C}}^{\neg }(1)$. Assume that $M:\mathcal{C}\xrightarrow{Y}Sh(\mathcal{C})\xrightarrow{i_*}Sh(B)$ is coherent. Write $\varphi :\mathcal{C}\to \widetilde{\mathcal{C}}$ for the completion of $\mathcal{C}$ under finite disjoint coproducts and $\widetilde{M}:\widetilde{\mathcal{C}}\to Sh(B)$ for the unique extension of $M$. Then we have isomorphisms of distributive lattices
\[\begin{tikzcd}
	& B \\
	{L_{\widetilde{\mathcal{C}}}^1(\Gamma \widetilde{M})} & {L_{\mathcal{C}}^1(\Gamma M)} & {Sub_{\mathcal{C}}(1)}
	\arrow["\psi"', from=1-2, to=2-1]
	\arrow[hook, from=1-2, to=2-3]
	\arrow["\cong", from=2-2, to=2-1]
	\arrow["{\cong }", from=2-3, to=2-2]
\end{tikzcd}\]
    making the diagram commute. ($\psi $ is the map from Proposition \ref{psi2}. It takes $b\in B$ to $[1\sqcup \emptyset \hookrightarrow 1\sqcup 1 ^{(\restr{*}{b},\restr{*}{\neg b})}]$.)
\end{proposition}

\begin{proof}
    By \cite[Theorem 8.4.2]{makkai} $\varphi $ is conservative, full wrt.~subobjects and any object of $\widetilde{\mathcal{C}}$ is the finite disjoint coproduct of objects coming from $\mathcal{C}$. To simplify notation, we may assume that $\varphi $ is injective on objects and therefore that $\widetilde{M}\varphi =M$. 
    
    By Proposition \ref{injsurj} the homomorphism $L^1((\varphi ,id)):L^1_{\mathcal{C}}(\Gamma M)\to L^1_{\widetilde{\mathcal{C}}}(\Gamma \widetilde{M})$, sending $[u\hookrightarrow x^s]$ to $[\varphi u\hookrightarrow \varphi x^s]$ is injective. We prove that it is surjective as well. 

    We will write $x$ instead of $\varphi x$. Take $[u_1\sqcup \dots \sqcup u_n\hookrightarrow x_1\sqcup \dots \sqcup x_n ^{s\in \widetilde{M}(\bigsqcup x_i)(\top )}]\in L^1(\Gamma \widetilde{M})$. $s$ corresponds to a map $s':\widehat{\top }\to \bigsqcup Mx_i $ which we can write as $\bigsqcup \widehat{b_i}\xrightarrow{\bigsqcup s'_i} \bigsqcup Mx_i $. By the definition of $M$ each (corresponding) $s_i\in Mx_i(b_i)$ is a $\mathcal{C}$-map $s_i:b_i\to x_i$. Note that $M(s_i)=s_i\circ - : M(b_i)=\restr{\mathcal{C}(-,b_i)}{B}=\widehat{b_i} \to M(x)=\restr{\mathcal{C}(-,x_i)}{B}$ sends $id_{b_i}$ to $s_i$. Therefore we have a commutative square 
    
\[\begin{tikzcd}
	{\widehat{b_i}=M(b_i)} && {M(x_i)} \\
	\\
	{\bigsqcup \widehat{b_i}=\widehat{\top}} && {\bigsqcup M(x_i)=\widetilde{M}(\bigsqcup x_i)}
	\arrow["{M(s_i)}", from=1-1, to=1-3]
	\arrow[from=1-1, to=3-1]
	\arrow[from=1-3, to=3-3]
	\arrow["{s'}"', from=3-1, to=3-3]
\end{tikzcd}\]
meaning that $\widetilde{M}(\bigsqcup s_i):\widetilde{M}(\bigsqcup b_i)=\widehat{\top }\to  \widetilde{M}(\bigsqcup x_i )$ takes $id_{\top }$ to $s$. Hence $(1 ,*)\xrightarrow{\bigsqcup s_i } (\bigsqcup x_i ,s)$ is a map in $(\int \Gamma \widetilde{M})$, and by taking pullback along it we get
\[
[u_1\sqcup \dots \sqcup u_n\hookrightarrow x_1\sqcup \dots \sqcup x_n \ ^{s}]= [(u_1\cap b_1)\cup \dots \cup (u_n\cap b_n)\hookrightarrow 1^*]
\]

This proves that $L^1((\varphi ,id))$ is an isomorphism and that the map $Sub_{\mathcal{C}}(1)\to L^1_{\mathcal{C}}(\Gamma M)$ defined by $u\mapsto [u\hookrightarrow 1^*]$ is surjective. But it is also injective: if some $w\hookrightarrow 1$ satisfies that $*\in \Gamma M(1)$ lies in $\Gamma M(w)$ then $id: 1\to 1$ factors through $w\hookrightarrow 1$ and therefore $w=1$. In particular $w\cap u=w\cap v$ implies $u=v$ and this is sufficient by Proposition \ref{siminL0}. Finally, the commutativity of the triangle follows by $[b\hookrightarrow 1^*]=[1\sqcup \emptyset \hookrightarrow 1\sqcup 1 ^{(\restr{*}{b},\restr{*}{\neg b})}]$ (we can take pullback along $b\sqcup \neg b \to 1\sqcup 1$).

\end{proof}

\begin{corollary}
\label{Yi*poscl}
    $\mathcal{C}, B=Sub^{\neg }_{\mathcal{C}}(1), M$ as above. Then $M$ is strongly positively closed iff $Sub_{\mathcal{C}}(1)$ is a Boolean algebra. 
    
    If $B$ is complete then $M$ is positively closed iff for every $u\in Sub_{\mathcal{C}}(1)$ there is a largest $b\leq u$ (with $b\in B$), moreover $b=\bigcap \{b'\in B: b'\geq u \}$ (in $B$).
\end{corollary}

\begin{proof}
    It is elementary to check that an arbitrary model $N:\mathcal{C}\to Sh(B)$ is (strongly) positively closed iff its extension $\widetilde{N}:\widetilde{\mathcal{C}}\to Sh(B)$ is (strongly) positively closed. Then the claim follows by Theorem \ref{strongposclequiv} and Theorem \ref{Bcompletestrongposcl}.
\end{proof}

\begin{example}
\label{Closed(X)}
    Let $X$ be an extremally disconnected Stone space. So it has the following properties: it is compact, Hausdorff, has a clopen basis and the closure of an open set is (cl)open. In other terms the interior of a closed set is clopen.

    Write $Closed(X)$ for the distributive lattice of closed sets in $X$ (with the usual operations). The sub-lattice (Boolean algebra) of complemented elements is $Clopen(X)$, which is complete (see e.g.~\cite[Theorem 39]{halmos}). 

    We claim that $int:\ Closed(X)\to Clopen(X)$ is a homomorphism. The only non-trivial thing to check is that $int(Z_1\cup Z_2)=int(Z_1)\cup int(Z_2)$ when $Z_1,Z_2$ are closed. $\supseteq $ is obvious. Also, since $int (Z_1\cup Z_2)\setminus Z_2$ is an open subset of $Z_1$ it must be contained in $int(Z_1)$, similarly $int (Z_1\cup Z_2)\setminus Z_1 \subseteq int (Z_2)$. Hence $int (Z_1\cup Z_2)\setminus (int (Z_1) \cup int(Z_2) )$ is a clopen subset of $Z_1\cap Z_2$. We got that 
    \[
    int (Z_1\cup Z_2)\setminus (int (Z_1) \cup int(Z_2) )\subseteq int (Z_1) \cup int(Z_2)
    \]
    meaning that the left hand side is $\emptyset$.

    This model $int:Closed(X)\to Clopen(X)\subseteq Sh(Clopen(X))=Sh(X)$ is positively closed. (Sheaves are taken wrt.~the finite union topology.) To see this, by Corollary \ref{Yi*poscl} we have to check that $int(Z)$ is the largest clopen set contained in all clopens that contain $Z$. This is clear, since $Z^c$ is the union of clopens, so $Z$ is the (set-theoretic) intersection of all clopens containing it, therefore $int(Z)$ must be their intersection in $Clopen(X)$.

    If $X$ is not discrete then $Clopen(X)\neq Closed(X)$. So in this case $int$ is positively closed while it is not strongly positively closed.
\end{example}

\section{Appendix: the proof of Theorem \ref{factorbeta}}

\begin{proposition}
\label{diagfilt}
    Let $\kappa $ be weakly compact, $\mathcal{C}$ a $\kappa$-coherent category with $\kappa $-small disjoint coproducts, $B$ a $\kappa$-coherent Boolean algebra, $M:\mathcal{C}\to Sh(B,\tau _{\kappa -coh})$ a $\kappa$-coherent functor, and $\beta :M\Rightarrow S_{\mathcal{C}}=S_{\mathcal{C}}^{Sh(B)}$ a natural transformation. 
    
    Let $Diag(M,\beta )$ be the following category: its objects are pairs $(u\hookrightarrow x,s)$ where $u\hookrightarrow x$ is a subobject in $\mathcal{C}$, $s\in Mx(\top )$ and $\beta _{x,\top }(s)(u)=\widehat{\top} =Id_{\kappa }(B)$. A morphism $(u\hookrightarrow x,s)\to (v\hookrightarrow y,s')$ is a map $f:x\to y$ such that $\restr{f}{u}$ factors through $v$ and $Mf_{\top }(s)=s'$. 
    
    Then $Diag(M,\beta )^{op}$ is $\kappa $-filtered.
\end{proposition}

\begin{proof}
It is non-empty, $(1\hookrightarrow 1, *\in M1(\top ))$ is a terminal object in $Diag(M,\beta )$.

    A cone (in $Diag(M,\beta )$) over $<\kappa $ objects $(u_i\hookrightarrow x_i ,\  s_i\in Mx_i(\top ))_{i<\gamma <\kappa }$ is given by the product:

\[\begin{tikzcd}
	{\prod _i u_i} & {\prod _i x_i} & {\prod _i Mx_i (\top )} & {\langle s_i \rangle } \\
	{u_j} & {x_j} & {Mx_j(\top )} & {s_i}
	\arrow[hook, from=1-1, to=1-2]
	\arrow[from=1-1, to=2-1]
	\arrow["{\pi _j}", from=1-2, to=2-2]
	\arrow["\ni"{description}, draw=none, from=1-3, to=1-4]
	\arrow[from=1-3, to=2-3]
	\arrow[shorten <=2pt, shorten >=2pt, maps to, from=1-4, to=2-4]
	\arrow[hook, from=2-1, to=2-2]
	\arrow["\ni"{description}, draw=none, from=2-4, to=2-3]
\end{tikzcd}\]
Indeed, we have to check $\beta _{\prod x_i , \top}(\langle  s_i\rangle _i)(\prod u_i)=\widehat{\top }$. Note that $\prod u_i = u_0\times u_1\times \dots =(u_0\times x_1\times x_2\times  \dots )\cap (x_0\times u_1\times x_2\times \dots )\cap \dots =\bigcap _i \pi _i^{-1}(u_i)$. Since $\beta _{\prod x_i , \top}(\langle  s_i\rangle _i)$ preserves $\kappa $-small meets it suffices to check $\beta _{\prod x_i , \top}(\langle  s_i\rangle _i)(\pi _j^{-1}(u_j))=\widehat{\top }$. By naturality we have

\[\begin{tikzcd}
	{M(\prod _i x_i)(\top)} && {Mx_j(\top )} \\
	\\
	{\mathbf{Coh}_{\kappa ,\kappa }(Sub_{\mathcal{C}}(\prod _i x_i),Id_{\kappa }(B))} && {\mathbf{Coh}_{\kappa ,\kappa }(Sub_{\mathcal{C}}(x_j),Id_{\kappa }(B))}
	\arrow["{M(\pi _j)_{\top}}", from=1-1, to=1-3]
	\arrow["{\beta _{\prod x_i, \top }}"', from=1-1, to=3-1]
	\arrow["{\beta _{x_j, \top }}", from=1-3, to=3-3]
	\arrow["{-\circ \pi _j^{-1}}"', from=3-1, to=3-3]
\end{tikzcd}\]
meaning $\beta _{x_j,\top }(s_j)$ equals $Sub_{\mathcal{C}}(x_j)\xrightarrow{\pi _j^{-1}} Sub_{\mathcal{C}}(\prod _i x_i)\xrightarrow{\beta _{\prod x_i ,\top }(\langle s_i \rangle )}Id_{\kappa }(B)$ and since the composite takes $u_j$ to $\widehat{\top }$ the second map takes $\pi _j^{-1}(u_j)$ to $\widehat{\top }$.

To find a cone over a general $\kappa $-small diagram, it suffices to find a cone over the objects (say $(f_i:x\to y_i)_i$), then to find an arrow $x'\to x$ which equalizes $x\xrightarrow{f_i}y_i\xrightarrow{h}y_j$ and $f_j$ for each arrow $h$ in the diagram. So we are left with the case of a joint equalizer.

Given $<\kappa $ parallel pairs $(f_i,g_i:(u\hookrightarrow x,s)\to (v_i\hookrightarrow y_i,t_i))_{i<\gamma <\kappa }$ there is a cone

\[\begin{tikzcd}
	{u'} & e & {Me(\top )} & s \\
	u & x & {Mx(\top )} & s \\
	{v_i} & {y_i} & {My_i(\top )} & {t_i}
	\arrow[""{name=0, anchor=center, inner sep=0}, hook, from=1-1, to=1-2]
	\arrow[hook', from=1-1, to=2-1]
	\arrow["eq", hook', from=1-2, to=2-2]
	\arrow[hook', from=1-3, to=2-3]
	\arrow["{\ni }"{description}, draw=none, from=1-4, to=1-3]
	\arrow[""{name=1, anchor=center, inner sep=0}, hook, from=2-1, to=2-2]
	\arrow[shift right, dashed, from=2-1, to=3-1]
	\arrow[shift left, dashed, from=2-1, to=3-1]
	\arrow["{f_i}"', shift right, from=2-2, to=3-2]
	\arrow["{g_{i}}", shift left, from=2-2, to=3-2]
	\arrow[shift right, from=2-3, to=3-3]
	\arrow[shift left, from=2-3, to=3-3]
	\arrow["\ni"{description}, draw=none, from=2-4, to=2-3]
	\arrow[hook, from=3-1, to=3-2]
	\arrow["\ni"{description}, draw=none, from=3-4, to=3-3]
	\arrow["pb"{description}, draw=none, from=0, to=1]
\end{tikzcd}\]
(where $pb$ means pullback, $eq$ means joint equalizer). Note that $u'\hookrightarrow u$ equalizes the dashed arrows as $v_i\hookrightarrow y_i$ is mono. So it remains to check $\beta _{e,\top }(s)(u')=\widehat{\top }$. As before, naturality implies that $\beta _{x,\top }(s)$ can be written as $Sub_{\mathcal{C}}(x)\xrightarrow{-\cap e} Sub_{\mathcal{C}}(e)\xrightarrow{\beta _{e,\top }(s)}Id_{\kappa }(B)$, and since the composite takes $u$ to $\widehat{\top}$, the second map takes $u'$ to $\widehat{\top }$.

\end{proof}

We will construct $\mathcal{C}_{M,\beta }$ as a $\kappa $-filtered (2,1)-colimit of slices along $Diag(M,\beta )$. An explicit construction of filtered (2,1)-colimits of categories can be found in \cite{2filtered}. An overview was given in \cite[Theorem 3.4]{bvalued}.

\begin{proposition}
\label{CMbeta}
     $\kappa ,\mathcal{C},B,M,\beta $ as before. 
     
     Write $\mathcal{C}_{M,\beta }$ for the $\kappa $-filtered $(2,1)$-colimit $\colim _{(u\hookrightarrow x,s)\in Diag(M,\beta )^{op}} \faktor{\mathcal{C}}{u}$ and write $\varphi :\mathcal{C}\to \mathcal{C}_{M,\beta }$ for the cocone map. Then $\mathcal{C}_{M,\beta }$ is $\kappa$-coherent with $\kappa $-small disjoint coproducts and $\varphi $ is $\kappa$-coherent.

     Moreover, the map $\psi: B\to \mathcal{C}_{M,\beta }$ sending $b\hookrightarrow \top$ to the equivalence class of
\[\begin{tikzcd}
	{1\sqcup \emptyset } && {1\sqcup 1} \\
	& {1\sqcup 1}
	\arrow[hook, from=1-1, to=1-3]
	\arrow[from=1-1, to=2-2]
	\arrow[from=1-3, to=2-2]
\end{tikzcd}\]
    living in the slice $\faktor{\mathcal{C}}{1\sqcup 1}$ indexed by $(1\sqcup 1 \hookrightarrow 1\sqcup 1, (\restr{*}{b},\restr{*}{\neg b}))$, is $\kappa$-coherent.

\end{proposition}

\begin{proof}
    The first part of the claim follows from \cite[Theorem 4.1.i)]{bvalued}. 
    
    We prove that $\psi $ preserves $\top ,\bot ,\neg $ and $\kappa $-small $\bigwedge $. The case of $\top $ and $\bot $ is clear. 

    $\neg $ is preserved: $(\iota _2,\iota _1):(1\sqcup 1\hookrightarrow 1\sqcup 1, (\restr{*}{b},\restr{*}{\neg b}))\to (1\sqcup 1\hookrightarrow 1\sqcup 1, (\restr{*}{\neg b},\restr{*}{b}))$ is a map in $Diag(M,\beta )$, therefore $\psi (\neg b)$ is equivalent to the class of $\emptyset \sqcup 1\hookrightarrow 1\sqcup 1$ living in $\faktor{\mathcal{C}}{1\sqcup 1}$ indexed by $(1\sqcup 1\hookrightarrow 1\sqcup 1, (\restr{*}{ b},\restr{*}{\neg b}))$, which is the complement of $1\sqcup \emptyset  = \psi (b)$.

    $\bigwedge $ is preserved: Let $(b_i)_{i<\gamma <\kappa }$ be a collection of elements in $B$. Given $\varepsilon :\gamma \to \{+,-\}$ we write $b_{\varepsilon }$ for $\bigwedge _{i} b_i^{\varepsilon (i)}$. As $\kappa $ is weakly compact $2^{\gamma }<\kappa $ so the collection $(b_{\varepsilon })_{\varepsilon \in 2^{\gamma }}$ is $\kappa $-small.

    For any fixed $i\in \gamma $, the map $(\iota _{\varepsilon (i)})_{\varepsilon }: \bigsqcup _{\varepsilon \in 2^{\gamma }} 1 \to 1\sqcup  1$ sends the $\varepsilon ^{\text{th}}$ copy of $1$ to $1\sqcup \emptyset $ if $\varepsilon (i)=+$ and to $\emptyset \sqcup 1$ if $\varepsilon (i)=-$. This yields a map 
    \[(\bigsqcup _{\varepsilon }1 \hookrightarrow \bigsqcup _{\varepsilon }1, (\restr{*}{b_{\varepsilon}})_{\varepsilon })\to (1\sqcup 1\hookrightarrow 1\sqcup 1 , (\restr{*}{b_i},\restr{*}{\neg b_i}))\]
    in $Diag(M,\beta )$. Hence $\psi (b_i)$ is equivalent to $(\bigsqcup _{\varepsilon :\  \varepsilon (i)=+}1\hookrightarrow \bigsqcup _{\varepsilon } 1, (\restr{*}{b_{\varepsilon}})_{\varepsilon })$. Their intersection is therefore represented by $(1\xhookrightarrow{ \iota _{\varepsilon ^+}}\bigsqcup _{\varepsilon } 1, (\restr{*}{b_{\varepsilon}})_{\varepsilon })$ where $\varepsilon ^+:\gamma \to \{+,-\}$ is the constant $+$ function.

    Similarly, the map $1\sqcup \bigsqcup _{\varepsilon \neq \varepsilon ^+}1 \to 1\sqcup 1$ yields a map 
    \[(1\sqcup \bigsqcup _{\varepsilon \neq \varepsilon ^+}1 \hookrightarrow 1\sqcup \bigsqcup _{\varepsilon \neq \varepsilon ^+}1, (\restr{*}{b_{\varepsilon}})_{\varepsilon })\to (1\sqcup 1\hookrightarrow 1\sqcup 1 , (\restr{*}{\bigwedge _i b_i},\restr{*}{\neg \bigwedge _i b_i}))\]
    in $Diag(M,\beta )$. By taking the pullback of $1\sqcup \emptyset $ we get $(1\xhookrightarrow{ \iota _{\varepsilon ^+}}\bigsqcup _{\varepsilon } 1, (\restr{*}{b_{\varepsilon}})_{\varepsilon })$ as a representation of $\psi (\bigwedge _i b_i)$ living in $\faktor{\mathcal{C}}{\bigsqcup _{\varepsilon }1}$, indexed by $(\bigsqcup _{\varepsilon }1 \hookrightarrow \bigsqcup _{\varepsilon }1, (\restr{*}{b_{\varepsilon}})_{\varepsilon })$. But this coincides with $\bigwedge _i \psi (b_i)$.
\end{proof}

\begin{proposition}
    We have a natural transformation $\delta :\psi ^*M \Rightarrow Y\varphi $
\[\begin{tikzcd}
	&& {Sh(B)} \\
	{\mathcal{C}} &&&& {Sh(\mathcal{C}_{M,\beta })} \\
	&& {\mathcal{C}_{M,\beta }}
	\arrow["{\psi ^*}", curve={height=-12pt}, from=1-3, to=2-5]
	\arrow["\delta", shorten <=6pt, shorten >=6pt, Rightarrow, from=1-3, to=3-3]
	\arrow["M", curve={height=-12pt}, from=2-1, to=1-3]
	\arrow["{\varphi }"', curve={height=12pt}, from=2-1, to=3-3]
	\arrow["Y"', curve={height=12pt}, from=3-3, to=2-5]
\end{tikzcd}\]
    defined as follows: given $x\in \mathcal{C}$ the component $\delta _x$ is induced by the cocone
    
\[\begin{tikzcd}
	& {\widehat{\varphi (x)}} \\
	\\
	& {\colim =\psi ^*Mx} \\
	{\widehat{\psi (b)} ^{\ s\in Mx(b)}} & \dots
	\arrow["{\delta _x}", from=3-2, to=1-2]
	\arrow[curve={height=-18pt}, from=4-1, to=1-2]
	\arrow[from=4-1, to=3-2]
	\arrow[from=4-1, to=4-2]
\end{tikzcd}\]
whose leg at vertex $(b,s)\in (\int Mx)^{op}$ is the $Y$-image of the map $\psi (b)\to \varphi (x)$ represented by 
\[\begin{tikzcd}
	{x\sqcup \emptyset } && {(x\times x) \sqcup x} \\
	& {x\sqcup 1 }
	\arrow["{\Delta \sqcup {!}}", from=1-1, to=1-3]
	\arrow["{1\sqcup {!}}"', from=1-1, to=2-2]
	\arrow["{\pi _2\sqcup {!}}", from=1-3, to=2-2]
\end{tikzcd}\]
living in $\faktor{\mathcal{C}}{x\sqcup 1}$ indexed by $(x\sqcup 1\hookrightarrow x\sqcup 1, (\restr{s}{b},\restr{*}{\neg b}))$.
    
\end{proposition}

\begin{proof}
    First we prove that $\delta _x$ is well-defined, i.e.~given $b\leq b'$ and $s'\in Mx(b')$, the triangle
\[\begin{tikzcd}
	& {\widehat{\varphi (x)}} \\
	\\
	{\widehat{\psi (b)} ^{\ \restr{s'}{b}\in Mx(b)}} && {\widehat{\psi (b')} ^{\ s'\in Mx(b')}}
	\arrow[from=3-1, to=1-2]
	\arrow[hook, from=3-1, to=3-3]
	\arrow[from=3-3, to=1-2]
\end{tikzcd}\]
    commutes. Indeed, all ingredients can be moved to the slice $\faktor{\mathcal{C}}{x\sqcup x\sqcup 1}$ indexed by $(x\sqcup x\sqcup 1\hookrightarrow x\sqcup x\sqcup 1, (\restr{s'}{b},\restr{s'}{b'\setminus b},\restr{*}{\neg b'}))$, along the routes

\[
\adjustbox{width=\textwidth}{
\begin{tikzcd}
	\textcolor{rgb,255:red,153;green,153;blue,153}{{\varphi (x)=(x\times x)\sqcup (x\times 1)}} && {x\sqcup x\sqcup 1^{(\restr{s'}{b},\restr{s'}{b'\setminus b},\restr{*}{\neg b'})}} && \textcolor{rgb,255:red,153;green,153;blue,153}{{(x\times x)\sqcup (x\times 1)=\varphi (x)}} \\
	\textcolor{rgb,255:red,153;green,153;blue,153}{{\psi (b)=x\sqcup \emptyset }} &&&& \textcolor{rgb,255:red,153;green,153;blue,153}{{x\sqcup \emptyset =\psi (b')}} \\
	& {x\sqcup 1^{(\restr{s'}{b},\restr{*}{\neg b})}} && {x\sqcup 1^{(\restr{s'}{b'},\restr{*}{\neg b'})}} \\
	{1\sqcup 1^{(\restr{*}{b},\restr{*}{\neg b})}} && {1^*} && {1\sqcup 1^{(\restr{*}{b'},\restr{*}{\neg b'})}} \\
	\textcolor{rgb,255:red,153;green,153;blue,153}{{\psi (b)=1\sqcup \emptyset}} && \textcolor{rgb,255:red,153;green,153;blue,153}{{\varphi (x)=x}} && \textcolor{rgb,255:red,153;green,153;blue,153}{{1\sqcup \emptyset = \psi (b')}}
	\arrow["{\pi _2\sqcup \pi _2}", color={rgb,255:red,153;green,153;blue,153}, from=1-1, to=3-2]
	\arrow["{id_x \sqcup ({!},{!})}"{description}, from=1-3, to=3-2]
	\arrow["{(id_x,id_x)\sqcup {!}}"{description}, from=1-3, to=3-4]
	\arrow["{\pi _2\sqcup \pi _2}"', color={rgb,255:red,153;green,153;blue,153}, from=1-5, to=3-4]
	\arrow["{\Delta \sqcup {!}}", color={rgb,255:red,153;green,153;blue,153}, from=2-1, to=1-1]
	\arrow[color={rgb,255:red,153;green,153;blue,153}, hook, from=2-1, to=3-2]
	\arrow["{\Delta \sqcup {!}}"', color={rgb,255:red,153;green,153;blue,153}, from=2-5, to=1-5]
	\arrow[color={rgb,255:red,153;green,153;blue,153}, from=2-5, to=3-4]
	\arrow[from=3-2, to=4-1]
	\arrow[from=3-2, to=4-3]
	\arrow[from=3-4, to=4-3]
	\arrow[from=3-4, to=4-5]
	\arrow[color={rgb,255:red,153;green,153;blue,153}, hook, from=5-1, to=4-1]
	\arrow[color={rgb,255:red,153;green,153;blue,153}, from=5-3, to=4-3]
	\arrow[color={rgb,255:red,153;green,153;blue,153}, hook, from=5-5, to=4-5]
\end{tikzcd}
}
\]
We claim (without proof) that by computing the pullbacks one obtains the triangle 
\[
\adjustbox{width=\textwidth}{
\begin{tikzcd}
	& \textcolor{rgb,255:red,153;green,153;blue,153}{{\psi (b')=x\sqcup x\sqcup \emptyset}} \\
	\\
	\textcolor{rgb,255:red,153;green,153;blue,153}{{\psi (b)=x\sqcup \emptyset \sqcup \emptyset }} && \textcolor{rgb,255:red,153;green,153;blue,153}{{\varphi (x)=(x\times x)\sqcup (x\times x)\sqcup (x\times 1)}} \\
	\\
	& {x\sqcup x\sqcup 1^{(\restr{s'}{b},\restr{s'}{b'\setminus b},\restr{*}{\neg b'})}}
	\arrow["{\Delta \sqcup \Delta \sqcup {!}}"{description}, color={rgb,255:red,153;green,153;blue,153}, from=1-2, to=3-3]
	\arrow[color={rgb,255:red,153;green,153;blue,153}, curve={height=6pt}, hook, from=1-2, to=5-2]
	\arrow[color={rgb,255:red,153;green,153;blue,153}, hook, from=3-1, to=1-2]
	\arrow["{\Delta \sqcup {!} \sqcup {!}}"{description, pos=0.7}, color={rgb,255:red,153;green,153;blue,153}, from=3-1, to=3-3]
	\arrow[color={rgb,255:red,153;green,153;blue,153}, curve={height=6pt}, hook, from=3-1, to=5-2]
	\arrow["{\pi _2\sqcup \pi _2 \sqcup \pi _2}"{description}, color={rgb,255:red,153;green,153;blue,153}, curve={height=-6pt}, from=3-3, to=5-2]
\end{tikzcd}
}
\]
which does commute.

Now we prove that $\delta $ is natural in $x$. Given $f:x\to y$ in $\mathcal{C}$ we need that the triangle
\[\begin{tikzcd}
	{\widehat{\psi (b)} ^{\ Mf_b(s)\in My(b)}} && {\widehat{\varphi (y)}} \\
	\\
	{\widehat{\psi (b)} ^{\ s\in Mx(b)}} && {\widehat{\varphi (x)}}
	\arrow[from=1-1, to=1-3]
	\arrow[equals, from=3-1, to=1-1]
	\arrow[from=3-1, to=3-3]
	\arrow["{\widehat{\varphi (f)}}"', from=3-3, to=1-3]
\end{tikzcd}\]
commutes. 

We can move all ingredients to the slice $\faktor{\mathcal{C}}{x\sqcup 1}$ indexed by $(x\sqcup 1 \hookrightarrow x\sqcup 1, (\restr{s}{b},\restr{*}{\neg b}))$ along 

\[
\adjustbox{width=\textwidth}{
\begin{tikzcd}
	\textcolor{rgb,255:red,153;green,153;blue,153}{{\varphi (x)=(x\times x)\sqcup (x\times 1)}} \\
	\textcolor{rgb,255:red,153;green,153;blue,153}{{\psi (b)=x\sqcup \emptyset }} && {x\sqcup 1^{(\restr{s}{b},\restr{*}{\neg b})}} && \textcolor{rgb,255:red,153;green,153;blue,153}{{(y\times y)\sqcup (y\times 1)=\varphi (y)}} \\
	&& {y\sqcup 1^{(\restr{Mf_b(s)}{b},\restr{*}{\neg b})}} && \textcolor{rgb,255:red,153;green,153;blue,153}{{y\sqcup \emptyset = \psi  (b)}} \\
	&& {1^*} & \textcolor{rgb,255:red,153;green,153;blue,153}{{y=\varphi (y)}} \\
	&&& \textcolor{rgb,255:red,153;green,153;blue,153}{{x=\varphi (x)}}
	\arrow["{\pi _2 \sqcup \pi _2}"{description}, color={rgb,255:red,153;green,153;blue,153}, from=1-1, to=2-3]
	\arrow["{\Delta \sqcup {!}}", color={rgb,255:red,153;green,153;blue,153}, from=2-1, to=1-1]
	\arrow[draw={rgb,255:red,153;green,153;blue,153}, hook, from=2-1, to=2-3]
	\arrow["{f\sqcup 1}"', from=2-3, to=3-3]
	\arrow["{\pi _2 \sqcup \pi _2}"{description}, color={rgb,255:red,153;green,153;blue,153}, from=2-5, to=3-3]
	\arrow[from=3-3, to=4-3]
	\arrow["{\Delta \sqcup {!}}"', color={rgb,255:red,153;green,153;blue,153}, from=3-5, to=2-5]
	\arrow[draw={rgb,255:red,153;green,153;blue,153}, hook', from=3-5, to=3-3]
	\arrow[draw={rgb,255:red,153;green,153;blue,153}, from=4-4, to=4-3]
	\arrow[draw={rgb,255:red,153;green,153;blue,153}, from=5-4, to=4-3]
	\arrow["{f=\varphi (f)}"', color={rgb,255:red,153;green,153;blue,153}, from=5-4, to=4-4]
\end{tikzcd}
}
\]

We claim (without proof) that once the pullbacks are computed, our triangle takes the form 

\[
\adjustbox{width=\textwidth}{
\begin{tikzcd}
	& \textcolor{rgb,255:red,153;green,153;blue,153}{{\varphi (x)=(x\times x)\sqcup (x\times 1)}} \\
	\textcolor{rgb,255:red,153;green,153;blue,153}{{\psi (b)=x\sqcup \emptyset }} &&& \textcolor{rgb,255:red,153;green,153;blue,153}{{\varphi (y)=(y\times x)\sqcup (y\times 1)}} \\
	\\
	& {x\sqcup 1^{(\restr{s}{b},\restr{*}{\neg b})}}
	\arrow["{\varphi (f)=(f\times id_x)\sqcup (f\times id_1)}"{description}, color={rgb,255:red,153;green,153;blue,153}, from=1-2, to=2-4]
	\arrow["{\pi _2 \sqcup \pi _2}"{description}, color={rgb,255:red,153;green,153;blue,153}, from=1-2, to=4-2]
	\arrow["{\Delta \sqcup {!}}"{description}, color={rgb,255:red,153;green,153;blue,153}, from=2-1, to=1-2]
	\arrow["{\langle f, id_x \rangle \sqcup {!}}"{description, pos=0.7}, color={rgb,255:red,153;green,153;blue,153}, from=2-1, to=2-4]
	\arrow[color={rgb,255:red,153;green,153;blue,153}, hook, from=2-1, to=4-2]
	\arrow["{\pi _2 \sqcup \pi _2}"{description}, color={rgb,255:red,153;green,153;blue,153}, from=2-4, to=4-2]
\end{tikzcd}
}
\]
which does commute. 
\end{proof}

\begin{proposition}
    Let $\kappa $ be weakly compact, and let $\mathcal{C}, \mathcal{D}, \mathcal{D}' $ be $\kappa$-coherent categories. Given a $\kappa$-coherent functor $F:\mathcal{D}\to \mathcal{D}'$, we have a natural transformation 
\[\begin{tikzcd}
	&& {Sh(\mathcal{D})} \\
	{\mathcal{C}} &&&& {Sh(\mathcal{D'})}
	\arrow["{F^*}", curve={height=-12pt}, from=1-3, to=2-5]
	\arrow["{S_{\mathcal{C}}^{Sh(\mathcal{D})}}", curve={height=-12pt}, from=2-1, to=1-3]
	\arrow[""{name=0, anchor=center, inner sep=0}, "{S_{\mathcal{C}}^{Sh(\mathcal{D'})}}"', curve={height=30pt}, from=2-1, to=2-5]
	\arrow["\nu", shorten <=7pt, shorten >=7pt, Rightarrow, from=1-3, to=0]
\end{tikzcd}\]
    induced by the cocone
    
\[\begin{tikzcd}
	& {\mathbf{DLat}_{\kappa ,\kappa }(Sub_{\mathcal{C}}(x),Sub_{Sh(\mathcal{D}')}(\widehat{\bullet }))} \\
	\\
	& {\colim =F ^*S_{\mathcal{C}}^{Sh(\mathcal{D})} (x)} \\
	{\widehat{F d }^{\ h\in \mathbf{DLat}_{\kappa ,\kappa }(Sub_{\mathcal{C}}(x),Sub_{Sh(\mathcal{D})}(\widehat{d}))}} & \dots
	\arrow["{\nu _x}", from=3-2, to=1-2]
	\arrow[curve={height=-18pt}, from=4-1, to=1-2]
	\arrow[from=4-1, to=3-2]
	\arrow[from=4-1, to=4-2]
\end{tikzcd}\]

    whose leg at vertex $(d, h: Sub_{\mathcal{C}}(x)\to Sub_{Sh(\mathcal{D})}(\widehat{d}))\in (\int S_{\mathcal{C}}^{Sh(\mathcal{D})}(x))^{op}$ sends $id_{F d}$ to $Sub_{\mathcal{C}}(x)\xrightarrow{h}Sub_{Sh(\mathcal{D})}(\widehat{ d})\xrightarrow{F^*} Sub_{Sh(\mathcal{D}')}(\widehat{F d})$.
\end{proposition}

\begin{proof}
    $\nu _x$ is well-defined, i.e.~given $r:(d_1,h_1)\to (d_2,h_2)$ in $(\int S_{\mathcal{C}}^{Sh(\mathcal{D})}(x))^{op}$ the triangle
    
\[\begin{tikzcd}
	& {\mathbf{DLat}_{\kappa ,\kappa }(Sub_{\mathcal{C}}(x),Sub_{Sh(\mathcal{D}')}(\widehat{\bullet }))} \\
	{\widehat{Fd_1}^{h_1}} && {\widehat{Fd_2}^{h_2}}
	\arrow[from=2-1, to=1-2]
	\arrow["{\widehat{Fr}}"', from=2-1, to=2-3]
	\arrow[from=2-3, to=1-2]
\end{tikzcd}\]
    commutes. But this is just the commutativity of 
\[\begin{tikzcd}
	&& {Sub_{Sh(\mathcal{D})}(\widehat{d_1})} && {Sub_{Sh(\mathcal{D'})}(\widehat{Fd_1})} \\
	{Sub_{\mathcal{C}}(x)} \\
	&& {Sub_{Sh(\mathcal{D})}(\widehat{d_2})} && {Sub_{Sh(\mathcal{D}')}(\widehat{Fd_2})}
	\arrow["{F^*}", from=1-3, to=1-5]
	\arrow["{h_1}", from=2-1, to=1-3]
	\arrow["{h_2}"', from=2-1, to=3-3]
	\arrow["{\widehat{r}^{-1}}"', from=3-3, to=1-3]
	\arrow["{F^*}"', from=3-3, to=3-5]
	\arrow["{\widehat{Fr}^{-1}}"', from=3-5, to=1-5]
\end{tikzcd}\]

$\nu $ is natural in $x$: take $f:x\to y$ in $\mathcal{C}$. By restricting the naturality square to a leg of the colimit cocone, we need the commutativity of
\[\begin{tikzcd}
	{\mathbf{DLat}_{\kappa ,\kappa }(Sub_{\mathcal{C}}(x),Sub_{Sh(\mathcal{D}')}(\widehat{\bullet }))} && {\mathbf{DLat}_{\kappa ,\kappa }(Sub_{\mathcal{C}}(y),Sub_{Sh(\mathcal{D}')}(\widehat{\bullet }))} \\
	{\widehat{Fd}^{h}} && {\widehat{Fd}^{h\circ f^{-1}}}
	\arrow["{(f^{-1})^{\circ}}", from=1-1, to=1-3]
	\arrow[from=2-1, to=1-1]
	\arrow[equals, from=2-1, to=2-3]
	\arrow[from=2-3, to=1-3]
\end{tikzcd}\]
which follows as both routes take $id_{Fd}$ to $Sub_{\mathcal{C}}(y)\xrightarrow{f^{-1}}Sub_{\mathcal{C}}(x)\xrightarrow{h}Sub_{Sh(\mathcal{D})}(\widehat{d})\xrightarrow{F^*}Sub_{Sh(\mathcal{D'})}(\widehat{Fd})$.

\end{proof}

\begin{remark}
\label{nositeversion}
    It is not necessary to fix a site presentation for $Sh(\mathcal{D}')$. Given a $\kappa$-coherent functor $F:\mathcal{D}\to Sh(\mathcal{D}')$ we can induce $\nu _x$ by the cocone whose leg at $(d,h)$ is 
    \[
    Fd=\widehat{Fd}=Sh(\mathcal{D})(-,Fd)\to S_{\mathcal{C}}^{Sh(\mathcal{D}')}(x)=\mathbf{DLat}_{\kappa ,\kappa }(Sub_{\mathcal{C}}(x),Sub_{Sh(\mathcal{D}')}(\bullet ))
    \]
    taking $id_{Fd}$ to $Sub_{\mathcal{C}}(x)\xrightarrow{h}Sub_{Sh(\mathcal{D})}(\widehat{d})\xrightarrow{F^*}Sub_{Sh(\mathcal{D}')}(Fd)$.
\end{remark}

\begin{theorem}
\label{chiimpliesN}
    $\kappa ,\mathcal{C},B,M,\beta $ as before. Assume that there is a $\kappa$-coherent functor $\chi :\mathcal{C}_{M,\beta } \to Sh(B)$ such that $\chi \circ \psi \geq Y $. This means, that for any $b\in B$ we have $\chi \circ \psi (b)\geq Y(b)$, as subobjects of $1\in Sh(B)$. Equivalently: there is a (necessarily unique) natural transformation $\mu :Y\Rightarrow \chi \circ \psi $. 
    
    Then there is a model $N:\mathcal{C}\to Sh(B)$ and a natural transformation $\alpha :M\Rightarrow N$ such that $\beta \leq  tp_N \circ \alpha $.
\end{theorem}

\begin{proof}
    The previously introduced ingredients fit together in the following diagram:
    
\[
\adjustbox{scale=0.95}{
\begin{tikzcd}
	&&&&&& {Sh(B)} \\
	\\
	\\
	{\mathcal{C}} &&&&&&& {Sh(\mathcal{C}_{M,\beta })} &&& {Sh(B)} \\
	&&&&&&& {}
	\arrow["{\psi ^*}"{description}, from=1-7, to=4-8]
	\arrow[""{name=0, anchor=center, inner sep=0}, equals, from=1-7, to=4-11]
	\arrow[""{name=1, anchor=center, inner sep=0}, "M"{description}, curve={height=-24pt}, from=4-1, to=1-7]
	\arrow[""{name=2, anchor=center, inner sep=0}, "{S_{\mathcal{C}}^{Sh(B)}}"{description}, curve={height=12pt}, from=4-1, to=1-7]
	\arrow[""{name=3, anchor=center, inner sep=0}, "{Y\varphi }"{description}, curve={height=-18pt}, from=4-1, to=4-8]
	\arrow[""{name=4, anchor=center, inner sep=0}, "{S_{\mathcal{C}}^{Sh(\mathcal{C}_{M,\beta })}}"{description}, curve={height=18pt}, from=4-1, to=4-8]
	\arrow["{S_{\mathcal{C}}^{Sh(B)}}"{description}, curve={height=60pt}, from=4-1, to=4-11]
	\arrow[""{name=5, anchor=center, inner sep=0}, "{\chi ^*}"{description}, from=4-8, to=4-11]
	\arrow["{\nu _2}", Rightarrow, from=4-8, to=5-8]
	\arrow["\delta"{description}, curve={height=6pt}, shorten <=8pt, shorten >=8pt, Rightarrow, from=1-7, to=3]
	\arrow["{\nu _1}"{description}, curve={height=-6pt}, shorten <=10pt, shorten >=10pt, Rightarrow, from=1-7, to=4]
	\arrow["{\widetilde{\mu }}"', shorten <=8pt, shorten >=8pt, Rightarrow, from=0, to=5]
	\arrow["\beta"', shorten <=5pt, shorten >=5pt, Rightarrow, from=1, to=2]
	\arrow["tp"', shorten <=5pt, shorten >=5pt, Rightarrow, from=3, to=4]
\end{tikzcd}
}
\]

Our candidates are: $N=\chi ^* Y \varphi $ and $\alpha = \chi ^* \delta \circ \widetilde{\mu }M$. 

First we will compare $\nu _1 \circ \psi ^* \beta $ and $tp \circ \delta $. Recall that $\psi ^*Mx=\colim _{(\int Mx)^{op}} \widehat{\psi b}$, so we will compute the two natural transformations out of $\psi ^*Mx$ modulo their restrictions to these representables (the \underline{claim} is that the restriction of  $\nu _1 \circ \psi ^* \beta $ is pointwise $\leq $ than the restriction of $tp \circ \delta $). We have a commutative diagram:

\[
\adjustbox{scale=0.9}{
\begin{tikzcd}
	&&&&&& {S_{\mathcal{C}}^{Sh(\mathcal{C}_{M,\beta})}x} \\
	&& {S_{\mathcal{C}}^{Sh(B)}x} &&&& {\psi ^*S_{\mathcal{C}}^{Sh(B)}x} \\
	&& Mx &&&& {colim=\psi ^*Mx} \\
	{\widehat{b}^s} &&&& {\widehat{\psi b}^s}
	\arrow["{({\nu _1})_x}"', from=2-7, to=1-7]
	\arrow["{\beta _x}"', from=3-3, to=2-3]
	\arrow["{\psi ^*\beta  _x}"', from=3-7, to=2-7]
	\arrow["{id_b\mapsto \beta _{x,b}(s)}", curve={height=-6pt}, from=4-1, to=2-3]
	\arrow["{id_b\mapsto s}"', from=4-1, to=3-3]
	\arrow["{id_{\psi b}\mapsto \psi ^* \circ \beta _{x,b}(s)}", curve={height=-18pt}, from=4-5, to=1-7]
	\arrow[curve={height=-6pt}, from=4-5, to=2-7]
	\arrow[from=4-5, to=3-7]
\end{tikzcd}
}
\]

That is: the map $\widehat{b}\xrightarrow{id_b\mapsto \beta _{x,b}(s)}S_{\mathcal{C}}^{Sh(B)}x$ is the $(b,\beta _{x,b}(s))$-leg of the canonical colimit (expressing $S_{\mathcal{C}}^{Sh(B)}x$ as a colimit of representables), so by the definition $\nu _1$, the $\psi ^*$-image post-composed with $(\nu _1)_x$ results $Sub_{\mathcal{C}}(x)\xrightarrow{\beta _{x,b}(s)}Sub_{Sh(B)}(\widehat{b})\xrightarrow{\psi ^*}Sub_{Sh(\mathcal{C}_{M,\beta })}(\widehat{\psi b})$. This homomorphism takes $u\hookrightarrow x$ to $\bigcup \{\psi b'\hookrightarrow \psi b : b'\in \beta _{x,b}(s)(u)\}$ (as $\beta _{x,b}(s)(u)\hookrightarrow \widehat{b}$ can be written as the union of all $\widehat{b'}$'s below it, and $\psi ^*$ preserves unions).

The other natural transformation is described as:

\[
\adjustbox{scale=0.75}{
\begin{tikzcd}
	{\widehat{\psi b}^s} && {\psi ^*Mx} && {Y\varphi x} && {S_{\mathcal{C}}^{Sh(\mathcal{C}_{M,\beta})}} \\
	{id_{\psi b}} &&&& {\Delta \sqcup {!}: \psi b\to \varphi x} && \begin{array}{c} u \mapsto \bigcup \makecell{\{d\hookrightarrow \psi b : d\hookrightarrow \psi b\to \varphi x \\ \text{ factors through } \varphi u \}} \end{array}
	\arrow[from=1-1, to=1-3]
	\arrow["{\delta _x}", from=1-3, to=1-5]
	\arrow["{tp_x}", from=1-5, to=1-7]
	\arrow[shorten <=12pt, shorten >=12pt, maps to, from=2-1, to=2-5]
	\arrow[shorten <=4pt, shorten >=4pt, maps to, from=2-5, to=2-7]
\end{tikzcd}
}
\]

So we have to prove that for $b'\leq b$ if $b'\in \beta _{x,b}(s)(u)$ then $\psi b'\to \psi b\to \varphi x$ factors through $\varphi u$.

We can represent $\psi b'\to \psi b\to \varphi x \leftarrow \varphi u$ in $\faktor{\mathcal{C}}{x\sqcup x\sqcup 1}$ indexed by $(x\sqcup x\sqcup 1 \hookrightarrow x\sqcup x\sqcup 1, (\restr{s}{b'}, \restr{s}{b\setminus b'}, \restr{*}{\neg b}))$ as:

\[
\adjustbox{scale=0.75}{
\begin{tikzcd}
	{x\sqcup \emptyset \sqcup \emptyset} &&&& {(u\times x)\sqcup (u\times x) \sqcup (u\times 1)} \\
	& {x\sqcup x \sqcup \emptyset} && {(x\times x) \sqcup (x\times x) \sqcup (x\times 1)} \\
	\\
	&& {x\sqcup x\sqcup 1}
	\arrow["{?}"{description}, dashed, from=1-1, to=1-5]
	\arrow[hook, from=1-1, to=2-2]
	\arrow[curve={height=30pt}, hook, from=1-1, to=4-3]
	\arrow[hook', from=1-5, to=2-4]
	\arrow["{\pi _2\sqcup \pi _2 \sqcup \pi _2}"{description}, curve={height=-30pt}, from=1-5, to=4-3]
	\arrow["{\Delta \sqcup \Delta \sqcup {!}}", from=2-2, to=2-4]
	\arrow[hook, from=2-2, to=4-3]
	\arrow["{\pi _2\sqcup \pi _2 \sqcup \pi _2}"{description}, from=2-4, to=4-3]
\end{tikzcd}
}
\]

If we are allowed to take the pullback of this diagram along $u\sqcup x\sqcup 1\hookrightarrow x\sqcup x\sqcup 1$ then the dashed arrow appears. So it suffices to prove that (assuming $b'\in \beta _{x,b}(s)(u)$) we have 
\[(u\sqcup x\sqcup 1\hookrightarrow x\sqcup x\sqcup 1,(\restr{s}{b'}, \restr{s}{b\setminus b'}, \restr{*}{\neg b}))\in Diag(M,\beta)\]
That is: $\beta _{x\sqcup x\sqcup 1, \top}((\restr{s}{b'}, \restr{s}{b\setminus b'}, \restr{*}{\neg b}))(u\sqcup x\sqcup 1 )=\widehat{\top }$.

Since $Sub_{Sh(B)}(\widehat{\top })=Sub_{Sh(B)}(\widehat{b'})\times Sub_{Sh(B)}(\widehat{b\setminus b'}) \times Sub_{Sh(B)}(\widehat{\neg b})$ it is enough to see that each composite 

\[\begin{tikzcd}
	&&&&& {Id_{\kappa }(\downarrow b' )} \\
	{Sub_{\mathcal{C}}(x\sqcup x\sqcup 1)} &&&& {Id_{\kappa }(\downarrow \top )} & {Id_{\kappa }(\downarrow b\setminus b')} \\
	&&&&& {Id_{\kappa }(\downarrow \neg b )}
	\arrow["{\beta _{x\sqcup x\sqcup 1, \top}((\restr{s}{b'}, \restr{s}{b\setminus b'}, \restr{*}{\neg b}))}", from=2-1, to=2-5]
	\arrow[from=2-5, to=1-6]
	\arrow[from=2-5, to=2-6]
	\arrow[from=2-5, to=3-6]
\end{tikzcd}\]

takes $u\sqcup x\sqcup 1$ to the top element. By naturality we have
\[\begin{tikzcd}
	{Sub_{\mathcal{C}}(x)} &&&& {Id_{\kappa }(\downarrow b' )} \\
	\\
	{Sub_{\mathcal{C}}(x\sqcup x\sqcup 1)} &&&& {Id_{\kappa }(\downarrow \top )}
	\arrow["{\beta _{x,b'}(\restr{s}{b'})}"', from=1-1, to=1-5]
	\arrow["{i_1^{-1}}"', from=3-1, to=1-1]
	\arrow["{\beta _{x\sqcup x\sqcup 1, \top}((\restr{s}{b'}, \restr{s}{b\setminus b'}, \restr{*}{\neg b}))}", from=3-1, to=3-5]
	\arrow[from=3-5, to=1-5]
\end{tikzcd}\]
and $b'\in \beta _{x,b}(s)(u)$ implies $b'\in \beta _{x,b'}(\restr{s}{b'})(u)$. The other two composites are trivial, and hence we proved the \underline{claim}.

Now we prove that the natural transformation
\[
\adjustbox{scale=0.9}{
\begin{tikzcd}
	&&&& {Sh(B)} \\
	\\
	&&&&& {Sh(\mathcal{C}_{M,\beta })} \\
	{\mathcal{C}} &&&&&&& {Sh(B)} \\
	&&&&& {}
	\arrow["{\psi ^*}"{description}, from=1-5, to=3-6]
	\arrow[""{name=0, anchor=center, inner sep=0}, curve={height=-24pt}, equals, from=1-5, to=4-8]
	\arrow[""{name=1, anchor=center, inner sep=0}, "{\chi ^*}"{description}, from=3-6, to=4-8]
	\arrow[""{name=2, anchor=center, inner sep=0}, "M"{description}, curve={height=-24pt}, from=4-1, to=1-5]
	\arrow[""{name=3, anchor=center, inner sep=0}, "{S_{\mathcal{C}}^{Sh(B)}}"{description, pos=0.6}, curve={height=12pt}, from=4-1, to=1-5]
	\arrow[""{name=4, anchor=center, inner sep=0}, "{S_{\mathcal{C}}^{Sh(\mathcal{C}_{M,\beta })}}"{description, pos=0.6}, curve={height=18pt}, from=4-1, to=3-6]
	\arrow[""{name=5, anchor=center, inner sep=0}, "{S_{\mathcal{C}}^{Sh(B)}}"{description, pos=0.6}, curve={height=50pt}, from=4-1, to=4-8]
	\arrow["{\widetilde{\mu }}"', shorten <=7pt, shorten >=7pt, Rightarrow, from=0, to=1]
	\arrow["\beta"', shorten <=5pt, shorten >=5pt, Rightarrow, from=2, to=3]
	\arrow["{\nu _1}"', shorten <=6pt, shorten >=6pt, Rightarrow, from=3, to=4]
	\arrow["{\nu _2}"', shorten <=7pt, shorten >=7pt, Rightarrow, from=4, to=5]
\end{tikzcd}
\tag{$*$}
}
\]
is pointwise $\leq $ than 

\[
\adjustbox{scale=0.9}{
\begin{tikzcd}
	&&&& {Sh(B)} \\
	\\
	\\
	{\mathcal{C}} &&&& {Sh(\mathcal{C}_{M,\beta })} &&& {Sh(B)}
	\arrow["{\psi ^*}"{description}, from=1-5, to=4-5]
	\arrow[""{name=0, anchor=center, inner sep=0}, curve={height=-18pt}, equals, from=1-5, to=4-8]
	\arrow["M"{description}, curve={height=-24pt}, from=4-1, to=1-5]
	\arrow[""{name=1, anchor=center, inner sep=0}, "{S_{\mathcal{C}}^{Sh(\mathcal{C}_{M,\beta })}}"{description}, curve={height=24pt}, from=4-1, to=4-5]
	\arrow[""{name=2, anchor=center, inner sep=0}, "{Y\varphi }"{description}, curve={height=-24pt}, from=4-1, to=4-5]
	\arrow[""{name=3, anchor=center, inner sep=0}, "{S_{\mathcal{C}}^{Sh(B)}}"{description}, curve={height=80pt}, from=4-1, to=4-8]
	\arrow[""{name=4, anchor=center, inner sep=0}, "{\chi ^*}"{description}, from=4-5, to=4-8]
	\arrow["\delta"{description}, curve={height=6pt}, shorten <=7pt, shorten >=7pt, Rightarrow, from=1-5, to=2]
	\arrow["{\widetilde{\mu }}"', shorten <=8pt, shorten >=8pt, Rightarrow, from=0, to=4]
	\arrow["tp"', shorten <=6pt, shorten >=6pt, Rightarrow, from=2, to=1]
	\arrow["{\nu _2}", shorten >=4pt, Rightarrow, from=4-5, to=3]
\end{tikzcd}
\tag{$**$}
}
\]

To compute the action of $\nu _2$ (and $\mu $) observe

\[
\adjustbox{width=\textwidth}{
\begin{tikzcd}
	&&&&&&&& {S_{\mathcal{C}}^{Sh(B)} x} \\
	&&& {S_{\mathcal{C}}^{Sh(\mathcal{C}_{M,\beta })}x} \\
	&&&&&&&& {\chi ^* S_{\mathcal{C}}^{Sh(\mathcal{C}_{M,\beta })}x} \\
	\\
	&&&& {colim=\psi ^*Mx} &&&& {\chi ^*\psi ^*Mx} \\
	{\widehat{\psi b}^s} &&&&& {\chi ^*\psi ^*\widehat{b}=\widehat{\chi \psi b}^s} &&& Mx \\
	&&&&& {\widehat{b}}
	\arrow["{(\nu _2)_x}"', from=3-9, to=1-9]
	\arrow["{(\nu _1\circ \psi ^*\beta )_x}"{description}, color={rgb,255:red,128;green,128;blue,128}, curve={height=-18pt}, from=5-5, to=2-4]
	\arrow["{(tp \circ \delta )_x}"{description}, curve={height=18pt}, from=5-5, to=2-4]
	\arrow["{\chi ^*(\nu _1\circ \psi ^*\beta )_x}", color={rgb,255:red,128;green,128;blue,128}, curve={height=-12pt}, from=5-9, to=3-9]
	\arrow["{\chi ^*(tp \circ \delta )_x}"', curve={height=12pt}, from=5-9, to=3-9]
	\arrow["{1_{\psi b}\mapsto \psi ^* \circ \beta _{x,b}(s)}"{description}, color={rgb,255:red,128;green,128;blue,128}, curve={height=-18pt}, from=6-1, to=2-4]
	\arrow["{1_{\psi b}\mapsto tp_{x,\psi b}(\Delta \sqcup \ !)}"{description}, curve={height=18pt}, from=6-1, to=2-4]
	\arrow[from=6-1, to=5-5]
	\arrow["{1_{\chi \psi b}\mapsto \chi ^* \circ \psi ^* \circ \beta _{x,b}(s)}", color={rgb,255:red,128;green,128;blue,128}, curve={height=-30pt}, from=6-6, to=1-9]
	\arrow["{1_{\chi \psi b}\mapsto \chi ^*\circ  tp_{x,\psi b}(\Delta \sqcup \ !)}"{description, pos=0.3}, curve={height=-12pt}, from=6-6, to=1-9]
	\arrow[from=6-6, to=5-9]
	\arrow["{\widetilde{\mu }_{Mx}}", from=6-9, to=5-9]
	\arrow["{\mu _b}", from=7-6, to=6-6]
	\arrow["{id_b\mapsto s}"', from=7-6, to=6-9]
\end{tikzcd}
}
\]
(Here we use the version of Remark \ref{nositeversion}, i.e.~we think of $\chi \psi b\in Sh(B)$ as the representable $Sh(B)(-,\chi \psi b)$.) 

We conclude that the natural transformation $(*)$ (resp.~$(**)$) takes $s\in Mx(b)$ to the upper (resp.~lower) homomorphism in 
\[
\adjustbox{width=\textwidth}{
\begin{tikzcd}
	&& {Sub_{Sh(B)}(\widehat{b})} \\
	{Sub_{\mathcal{C}}(x)} &&&& {Sub_{Sh(\mathcal{C}_{M,\beta })}(\widehat{\psi b})} & {Sub_{Sh(B)}(\chi ^*\psi ^*\ \widehat{b})} & {Sub_{Sh(B)}(\widehat{b})}
	\arrow["{\psi ^*}", curve={height=-6pt}, from=1-3, to=2-5]
	\arrow["{\beta _{x,b}(s)}", curve={height=-6pt}, from=2-1, to=1-3]
	\arrow[""{name=0, anchor=center, inner sep=0}, "{tp_{x,\psi b}(\Delta \sqcup  {!})}"', curve={height=12pt}, from=2-1, to=2-5]
	\arrow["{\chi ^*}", from=2-5, to=2-6]
	\arrow["{\mu _b^{-1}}", from=2-6, to=2-7]
	\arrow["{\leq }"{marking, allow upside down}, draw=none, from=1-3, to=0]
\end{tikzcd}
}
\]
This proves $(*)\leq (**)$. Moreover $\beta \leq (*)$ as we have
\[\begin{tikzcd}
	{Sub_{Sh(B)}(\widehat{b})} &&& {Sub_{Sh(B)}(\widehat{b})} \\
	& {Sub_{Sh(\mathcal{C}_{M,\beta })}(\widehat{\psi b})} & {Sub_{Sh(B)}(\chi ^*\psi ^*\ \widehat{b})}
	\arrow[""{name=0, anchor=center, inner sep=0}, curve={height=-18pt}, equals, from=1-1, to=1-4]
	\arrow["{\psi ^*}", from=1-1, to=2-2]
	\arrow[""{name=1, anchor=center, inner sep=0}, "{\chi ^*}", from=2-2, to=2-3]
	\arrow["{\mu _b^{-1}}", from=2-3, to=1-4]
	\arrow["{\leq }"{marking, allow upside down}, draw=none, from=0, to=1]
\end{tikzcd}\]
by the existence of the dashed arrow in

\[\begin{tikzcd}
	& {\chi ^*\psi ^*\widehat{b'}} && {\chi ^*\psi ^*\widehat{b}} \\
	&& pb \\
	& {\mu _b^{-1}\chi ^*\psi ^* \widehat{b'}} && {\widehat{b}} \\
	{\widehat{b'}}
	\arrow[hook, from=1-2, to=1-4]
	\arrow[hook, from=3-2, to=1-2]
	\arrow[hook, from=3-2, to=3-4]
	\arrow["{\mu _b}"', hook, from=3-4, to=1-4]
	\arrow["{\mu _{b'}}", hook, from=4-1, to=1-2]
	\arrow[dashed, from=4-1, to=3-2]
	\arrow[hook, from=4-1, to=3-4]
\end{tikzcd}\]
(and it is enough to check $\leq $ for representable subobjects as the two homomorphisms in question are union-preserving).

It remains to prove $\nu _2 \circ \chi ^* tp_{Y\varphi } =tp_{\chi ^*Y\varphi }$. Similarly as before, if we precompose with a leg of the colimit cocone: $\chi ^*\widehat{d} \xrightarrow{\chi ^*(1_d\mapsto s)} \chi ^*Y\varphi x \xrightarrow{(\nu _2 \circ \chi ^*tp_{Y\varphi})} S_{\mathcal{C}}^{Sh(B)}x$ we get 
\[
\widehat{b}\xrightarrow{t}\chi ^*\widehat{d} \ \ \mapsto \   Sub_{\mathcal{C}}(x)\xrightarrow{tp_{{Y\varphi},x,d}(s)}Sub_{Sh(\mathcal{C}_{M,\beta })}(\widehat{d})\xrightarrow{\chi ^*}Sub_{Sh(B)}(\chi ^*\widehat{d})\xrightarrow{t^{-1}}Sub_{Sh(B)}(\widehat{b})
\]

Here $tp _{Y\varphi ,x,d}(s)$ takes $u\hookrightarrow x$ to the pullback of $Y\varphi u\hookrightarrow Y\varphi x$ along $s:\widehat{d}\to Y\varphi x$.

The restriction of $tp_{\chi ^*Y\varphi }$ takes $\widehat{b}\xrightarrow{t}\chi ^*\widehat{d}$ to the map which takes $u\hookrightarrow x$ to the pullback of $\chi ^*Y\varphi u\hookrightarrow \chi ^*Y\varphi x$ along $\widehat{b}\xrightarrow{t} \chi ^*\widehat{d}\xrightarrow{\chi ^*s}\chi ^*Y\varphi x$. 

It follows that the two $Sub_{\mathcal{C}}(x)\to Sub_{Sh(B)}(\widehat{b})$ homomorphisms are the same, as $\chi ^*$ preserves pullbacks. This is enough, as the legs of the colimit cocone are jointly epic.     
\end{proof}

So the question is, whether $Y:B\to Sh(B)$ extends along $\psi :B\to \mathcal{C}_{M,\beta }$ (up to $\leq $). At least it extends along $\psi : B\to Sub_{\mathcal{C}_{M,\beta }}(1)$ (up to equality):

\begin{proposition}
\label{extendsfromB}
   $\kappa ,\mathcal{C},B,M,\beta $ as before. We have a commutative triangle:
\[\begin{tikzcd}
	B && {Sh(B)} \\
	\\
	{Sub_{\mathcal{C}_{M,\beta }}(1)=\colim _{(u\hookrightarrow x,s)\in Diag(M,\beta )^{op} } Sub_{\mathcal{C}}(u)}
	\arrow["Y", from=1-1, to=1-3]
	\arrow["\psi"', from=1-1, to=3-1]
	\arrow["{\chi _0}"', from=3-1, to=1-3]
\end{tikzcd}\]    
where $\chi _0$ is $\kappa$-coherent.
\end{proposition}

\begin{proof}
    At $(u\hookrightarrow x,s)$ we define $\chi _0$ to be $Sub_{\mathcal{C}}(u)\hookrightarrow Sub_{\mathcal{C}}(x)\xrightarrow{\beta _{x,\top }(s)}Sh(B)$. The first map preserves everything except the top object, but since $\beta _{x,\top }(s)(u)=\widehat{\top }$ the composite is $\kappa$-coherent. We have to check that these maps are compatible and that the resulting homomorphism makes the triangle commute.

    Compatibility: Let $h:x\to y$ be a morphism $(u\hookrightarrow x,s)\to (v\hookrightarrow y,t)$ in $Diag(M,\beta )$ (that is, $\restr{h}{u}$ factors through $v$ and $Mh_{\top }(s)=t$). Write $h':v^*\to v$ for the pullback of $h$ along $v\hookrightarrow y$. We have a diagram
\[\begin{tikzcd}
	{Sub_{\mathcal{C}}(u)} & {Sub_{\mathcal{C}}(v^*)} && {Sub_{\mathcal{C}}(x)} \\
	&& {=} && {=} & {Sh(B)} \\
	& {Sub_{\mathcal{C}}(v)} && {Sub_{\mathcal{C}}(y)}
	\arrow[hook, from=1-1, to=1-2]
	\arrow[hook, from=1-2, to=1-4]
	\arrow["{\beta _{x,\top }(s)}", from=1-4, to=2-6]
	\arrow["{(\restr{h}{})^{-1}}", from=3-2, to=1-1]
	\arrow["{(h')^{-1}}"', from=3-2, to=1-2]
	\arrow[hook, from=3-2, to=3-4]
	\arrow["{h^{-1}}"', from=3-4, to=1-4]
	\arrow["{\beta _{y,\top }(t)}"', from=3-4, to=2-6]
\end{tikzcd}\]
where everything commutes except the triangle on the left. It suffices to prove that it gets equalized by the map $Sub_{\mathcal{C}}(v^*)\to Sh(B)$. But it is equalized by $-\cap u: Sub_{\mathcal{C}}(v^*)\to Sub_{\mathcal{C}}(u)$, and $Sub_{\mathcal{C}}(v^*)\to Sh(B)$ factors through $-\cap u$ (as for $w\leq v^*$ we have $\beta _{x,\top }(s)(w\cap u)=\beta _{x,\top }(s)(w)$).

Commutativity: $\psi (b)= (1\sqcup \emptyset \hookrightarrow 1\sqcup 1)\in Sub_{\mathcal{C}}(1\sqcup 1)^{(1\sqcup 1\hookrightarrow 1\sqcup 1,(\restr{*}{b},\restr{*}{\neg b}))}$, so it is taken to $\beta _{1\sqcup 1,\top }(\restr{*}{b},\restr{*}{\neg b})(1\sqcup \emptyset)\in Id_{\kappa }(B)$ by $\chi _0$. We need to show that it equals $\downarrow b$. The diagram 

\[\begin{tikzcd}
	{Id_{\kappa }(\downarrow b)} && {Id_{\kappa }(B)} && {Id_{\kappa }(\downarrow \neg b)} \\
	\\
	{Sub_{\mathcal{C}}(1\sqcup \emptyset )} && {Sub_{\mathcal{C}}(1\sqcup 1)} && {Sub_{\mathcal{C}}(\emptyset \sqcup 1)}
	\arrow[from=1-3, to=1-1]
	\arrow[from=1-3, to=1-5]
	\arrow["{\beta _{1,b}(\restr{*}{b})}", from=3-1, to=1-1]
	\arrow["{\beta _{1\sqcup 1,b}(\restr{*}{b},-)}"{description, pos=0.3}, color={rgb,255:red,128;green,128;blue,128}, from=3-3, to=1-1]
	\arrow["{\beta _{1\sqcup 1,\top}((\restr{*}{b},\restr{*}{\neg b}))}"{description, pos=0.6}, from=3-3, to=1-3]
	\arrow["{\beta _{1\sqcup 1,\neg b}(-,\restr{*}{\neg b})}"{description, pos=0.3}, color={rgb,255:red,128;green,128;blue,128}, from=3-3, to=1-5]
	\arrow[from=3-3, to=3-1]
	\arrow[from=3-3, to=3-5]
	\arrow["{\beta _{1,\neg b}(\restr{*}{\neg b})}"', from=3-5, to=1-5]
\end{tikzcd}\]
    commutes, as $\beta _{x,b}$ is natural both in $x$ and in $b$. Therefore $\beta _{1\sqcup 1,\top }((\restr{*}{b},\restr{*}{\neg b}))(1\sqcup \emptyset )$ equals $\downarrow b$ (since its intersection with $\downarrow b$ is $\downarrow b$, and its intersection with $\downarrow \neg b$ is $\emptyset $).

\end{proof}

\begin{usethmcounterof}{factorbeta}
    Let $\kappa $ be weakly compact, $\mathcal{C}$ a $\kappa $-coherent category with $\kappa $-small disjoint coproducts, $B$ a $\kappa $-coherent Boolean algebra, $M:\mathcal{C}\to Sh(B,\tau _{\kappa -coh})$ a $\kappa $-coherent functor and $\beta :M\Rightarrow S_{\mathcal{C}}^{Sh(B)}$ a natural transformation. Assume moreover, that $Sh(B,\tau _{\kappa -coh})$ can realize $\kappa $-types. 

    Then there is a $\kappa $-coherent functor $N:\mathcal{C}\to Sh(B,\tau _{\kappa -coh})$ together with a natural transformation $\alpha :M\Rightarrow N$, such that $\beta \leq tp_N\circ \alpha  $ (pointwise).

\end{usethmcounterof}

\begin{proof}
    By Theorem \ref{chiimpliesN} and Proposition \ref{extendsfromB}.
\end{proof}

\section*{Funding}

Supported by the Grant Agency of the Czech Republic (under the grant
22-02964S), and by Masaryk University (project MUNI/A/1457/2023).

\bibliographystyle{plain}
\bibliography{main.bib}

\end{document}